\documentclass[reqno]{amsart}
\usepackage{amssymb}
\usepackage{graphicx}

\usepackage[usenames, dvipsnames]{color}
\usepackage{verbatim}
\usepackage{mathrsfs}
\usepackage{bm}
\usepackage{cite}

\numberwithin{equation}{section}

\newtheorem{theorem}{Theorem}[section]
\newtheorem{corollary}[theorem]{Corollary}
\newtheorem{lemma}[theorem]{Lemma}
\newtheorem{prop}[theorem]{Proposition}

\theoremstyle{definition}
\newtheorem{remark}[theorem]{Remark}

\theoremstyle{definition}
\newtheorem{definition}[theorem]{Definition}

\theoremstyle{definition}

\theoremstyle{definition}

\makeatletter
\def\dashint{\operatorname%
{\,\,\text{\bf-}\kern-.98em\DOTSI\intop\ilimits@\!\!}}
\makeatother

\def\\det{\text{\det}}

\def\Xint#1{\mathchoice
 {\XXint\displaystyle\textstyle{#1}}%
 {\XXint\textstyle\scriptstyle{#1}}%
 {\XXint\scriptstyle\scriptscriptstyle{#1}}%
 {\XXint\scriptscriptstyle\scriptscriptstyle{#1}}%
 \!\int}
\def\XXint#1#2#3{{\setbox0=\hbox{$#1{#2#3}{\int}$}
  \vcenter{\hbox{$#2#3$}}\kern-.5\wd0}}

\def\dashint{\Xint-}

\def\.5{\frac{1}{2}}

\newcommand{\RN}[1]{%
  \textup{\uppercase\expandafter{\romannumeral#1}}%
}

\renewcommand{\epsilon}{\varepsilon}

\newcounter{marnote}

\begin{document}

\title[Nonuniqueness on the Navier-Stokes equation]{Nonuniqueness analysis on the Navier-Stokes equation in $C_{t}L^{q}$ space}

\author[C.X. Miao]{Changxing Miao}
\address[C.X. Miao] {Institute of Applied Physics and Computational Mathematics, P.O. Box 8009, Beijing, 100088, China.}
\email{miao\_changxing@iapcm.ac.cn}

\author[Z.W. Zhao]{Zhiwen Zhao}

\address[Z.W. Zhao]{School of Mathematics and Physics, University of Science and Technology Beijing, Beijing 100083, China.}
%\address{2. Beijing Computational Science Research Center, Beijing 100193, China.}
\email{zwzhao365@163.com}

%\footnote{}

\date{\today} % delete this line to display the current date

%%% BEGIN DOCUMENT

\maketitle
%\tableofcontents
\begin{abstract}
In the presence of any prescribed kinetic energy, we implement the intermittent convex integration scheme with $L^{q}$-normalized intermittent jets to give a direct proof for the existence of solution to the Navier-Stokes equation in $C_{t}L^{q}$ for some uniform $2<q\ll3$ without the help of interpolation inequality. The result shows the sharp nonuniqueness that there evolve infinite nontrivial weak solutions of the Navier-Stokes equation starting from zero initial data. Furthermore, we improve the regularity of solution to be of $C_{t}W^{\alpha,q}$ in virtue of the fractional Gagliardo-Nirenberg inequalities with some $0<\alpha\ll1$. More importantly, the proof framework provides a stepping stone for future progress on the method of intermittent convex integration due to the fact that $L^{q}$-normalized building blocks carry the threshold effect of the exponent $q$ arbitrarily close to the critical value $3$.
\end{abstract}

\maketitle
%\date{}
%\maketitle
%{\bf Abstract}

%{Keywords}

%\noindent{\bf{MSC numbers}}: {35K92; 35B40; 35B65.}

\section{Introduction and main results}%\label{intro}

The mathematical model frequently used to describe incompressible viscous  fluid flow is the following Navier-Stokes equation in the periodic case:  for $T>0$,
\begin{align}\label{ZQ001}
\begin{cases}
\partial_{t}u-\nu\Delta u+\mathrm{div}(u\otimes u)+\nabla p=0,\\
\mathrm{div}u=0,
\end{cases}\quad\mathrm{on}\;[0,T]\times\mathbb{T}^{3},	
\end{align}
where $\mathbb{T}^{3}:=\mathbb{R}^{3}/2\pi\mathbb{Z}^{3}$, $u:[0,T]\times\mathbb{T}^{3}\rightarrow\mathbb{R}^{3}$ and $p:[0,T]\times\mathbb{T}^{3}\rightarrow\mathbb{R}$, respectively, represent the unknown velocity and scalar pressure, $\nu>0$ denotes the inverse Reynolds number, which is also called the kinematic viscosity up to the fixed characteristic scale and characteristic velocity of the flow. Particularly in high Reynolds number regime (that is, when $\nu$ is sufficiently small), this model can be utilized to interpret turbulent flow characterized by chaotic changes in flow velocity and pressure. In fact, a lot of natural phenomena and engineering processes are closely linked to turbulence, see e.g. \cite{B1953,FMRT2001,F1995,MY196501,MY196502,T2001}. Quantitative analysis on occurrence and evolution mechanism of turbulence is critical to its application in the fields of environmental protection, industrial equipments, biotechnology and medical technology etc. 

It has been verified to be a tremendous accuracy in numerical experiments \cite{KIYIU2003,PKV2002,S1998} in terms of the fundamental ansatz of fully developed Kolmogorov's 1941 turbulence theory \cite{K194101,K194102,K194103} that there appears anomalous dissipation of energy in the infinite Reynolds number limit. This ansatz also indicates that the solutions of Navier-Stokes equation may only converge to weak solutions of the Euler equations in the vanishing viscosity limit rather than keeping uniformly smooth convergence. Onsager \cite{O1949} analyzed the reason of energy dissipation possibly caused by the fact that weak solutions of the Euler equations are not regular enough to preserve energy conservation. To be specific, Onsager conjectured the dichotomy phenomena that energy dissipation occurs in the case when their weak solutions belong to H\"{o}lder space with regularity exponent $\alpha\leq\frac{1}{3}$, while the energy is conserved for $\alpha>\frac{1}{3}$. Onsager's conjecture has now been demonstrated due to \cite{CWT1994,BDSV2019,I2018} except for the critical case $\alpha=\frac{1}{3}$. With regard to more previous related results, see \cite{BDIS2015,DS2013,DS2009,DS2014,S1993,S1997,S2000} and the references therein. These facts stimulate us to work within the framework of weak solutions for translating predictions provided by turbulence theory into strictly mathematical conclusions. We now give the definition for weak solution of \eqref{ZQ001} as follows. 
\begin{definition}
For $T>0$, a vector-valued function	$u\in L^{2}([0,T]\times\mathbb{T}^{3})$ is termed as a weak solution of \eqref{ZQ001} provided that for any $t\in[0,T]$, the velocity field $u(\cdot,t)$ has zero mean, is weakly divergence free, and 
\begin{align*}
\int_{0}^{T}\int_{\mathbb{T}^{3}}u\cdot(\partial_{t}\varphi+\nu\Delta\varphi+(u\cdot\nabla)\varphi)dxdt=0,	
\end{align*}
for any divergence free test function $\varphi\in C_{0}^{\infty}((0,T);C^{\infty}(\mathbb{T}^{3}))$.
\end{definition}	
This type of weak solutions are also called very weak solutions based on the fact that $L^{2}([0,T]\times\mathbb{T}^{3})$ is the minimal regularity space to ensure the validity for the definition of solution in the distributional sense. From the work of Fabes-Jones-Rivi\`{e}re \cite{FJR1972}, we see that up to possible redefinition of the solution on a negligible set of zero measure in space-time, the above solutions are actually equivalent to mild or Oseen solutions (see Chapter 6 in \cite{L2016}) subject to the following integral equation 
\begin{align*}
u(\cdot,t)=e^{\nu t\Delta}u(\cdot,0)+\int_{0}^{t}e^{\nu(t-s)\Delta}\mathbb{P}_{H}\mathrm{div}(u(\cdot,s)\otimes u(\cdot,s))ds,	
\end{align*}  
where $\mathbb{P}_{H}$ denotes the Leray projection  onto the divergence-free vector fields and $e^{t\Delta}$ is the heat semigroup. If we further require that $u\in C_{\mathrm{weak}}([0,T];L^{2}(\mathbb{T}^{3}))\cap L^{2}([0,T];\dot{H}^{1}(\mathbb{T}^{3}))$, it becomes the classical Leray-Hopf weak solution introduced by Leray \cite{L1934} in $\mathbb{R}^{3}$ and later Hopf \cite{H1951} in domains with boundary. These solutions satisfy the energy inequality
\begin{align*}
\|u(\cdot,t)\|_{L^{2}}^{2}+2\nu\int_{0}^{t}\|\nabla u(\cdot,s)\|_{L^{2}}^{2}ds\leq\|u(\cdot,0)\|_{L^{2}}^{2},\quad\text{for any }t\in[0,T],
\end{align*}	 
which  explicitly exhibits the information of two important physical quantities including the kinetic energy and the dissipation of energy such that  the study on Leray-Hopf weak solution  turns into the central topic in fluid mechanics. While global well-posedness with small initial data and local well-posedness for their solutions have been established based on classical contraction mapping arguments (see e.g. \cite{CMP1994,L2016,KT2001,K1984,FK1964}), the global well-posedness with general initial data still remains to be open and has been listed as one of seven \textit{Clay Millennium Prize} problems \cite{F2000}.  Regularity and uniqueness of Leray-Hopf solutions are shown to hold with the additional assumption that the solutions are bounded in $L_{t}^{p}L^{q}$ for $2/p+3/q\leq1$, which is termed as the Lady\v{z}enskaja-Prodi-Serrin criteria \cite{ISS2003,KL1957,P1959,S1962}. When this criteria is violated, it was famously conjectured by Lady\v{z}enskaja \cite{L1967} that there will appear non-uniqueness of Leray-Hopf solutions. Recently, Jia and \v{S}ver\'{a}k \cite{JS2015} demonstrated the non-uniqueness under a certain spectral assumption for the linearized Navier-Stokes operator. The subsequent work \cite{GS2023} completed by Guillod and \v{S}ver\'{a}k presented compelling numerical evidence that such assumption may be true, althought a strict mathematical proof is not yet given. 

Analogous to Leray-Hopf solutions, the Lady\v{z}enskaja-Prodi-Serrin criteria also holds for mild/Oseen solutions \cite{FJR1972,FLT2000,LM2001,L2002,K2006}. Besides the requirement of higher Sobolev norm, Leray-Hopf solutions possess some nice properties such as weak-strong uniqueness (see e.g. \cite{C2011,G2006,W2018}) by contrast with mild/Oseen solutions. These facts indicate that the nonuniqueness of Leray-Hopf solutions are more difficult to handle than mild/Oseen solutions. Buckmaster and Vicol \cite{BV2019} were the first to prove the nonuniqueness of mild/Oseen solutions in $C_{t}L^{2^{+}}$ with bounded kinetic energy by developing the method of intermittent convex integration, whose ideas were inspired by previous numerical observations on the intermittency in hydrodynamic turbulence, see e.g. \cite{F1995,MY196501}. It is worthwhile to point out that the core of this scheme lies in achieving the coupling of concentration and oscillation. In fact, in contrast to highly oscillated $C^{\alpha}_{x,t}$ convex integration scheme developed by De Lellis and Sz\'{e}kelyhidi \cite{DS2013}, which was motivated by the earlier techniques of Nash twists \cite{N1954} and Kuiper corrugations \cite{K1955} for the isometric embedding problem, the concentration parameter was further added in intermittent Beltrami flows \cite{BV2019} for the purpose of weakening the effect of linear viscosity term $-\nu\Delta u$ in Lebesgue spaces. This makes the nonlinear convective term $u\cdot\nabla u$ hold the dominant position and therefore leads to the nonuniqueness of solutions. The scheme was subsequently refined by Buckmaster, Colombo and Vicol \cite{BCV2022} to construct non-unique weak solutions with the Hausdorff dimension of singular sets in time strictly less than $1$. Cheskidov and Luo \cite{CL2022} discovered the effect of temporal intermittency and established the nonuniqueness of weak solutions in $L^{p}_{t}L^{\infty}$ with any $1<p<2$. In addition, intermittent convex integration schemes have been used to solve the nonuniqueness of solutions for other PDEs in fluid dynamics (see e.g. \cite{CL2021,MS2018,MS2019,BBV2020,BN2023,BMNV2023}) and the stationary Navier-Stokes equation \cite{L2019}. With regard to more detailed introduction for the development of convex integration scheme, see \cite{BV2021,BV201902}. 

Despite several attempts in aforementioned work, current convex integration schemes seem far from fulfilling the requirement of Leray-Hopf solutions. Recently, Albritton, Bru\'{e} and Colombo \cite{ABC2022} proved the non-uniqueness of Leray-Hopf solutions to the forced Navier-Stokes equation in the whole space by using a different method based on self-similarity and instability. It is worth emphasizing that the results in \cite{ABC2022} only hold for some special nontrivial body forces. It remains to be open for the general body forces, especially for the case of trivial body forces and nonzero initial data. As for the corresponding non-uniqueness in bounded domains, see their subsequent work  \cite{ABC202202}.

Different from intermittent building blocks of $L^{2}$ normalization developed in \cite{BV2019,BCV2022}, in this paper we introduce the general $L^{q}$-normalized intermittent jets to reinterpret the method of intermittent convex integration and present a direct proof for the nonuniqueness of weak solutions to the Navier-Stokes equation in $C_{t}L^{q}$ for some $q$ slightly greater than $2$ without the assistance of interpolation inequality. Of greater importance is that the proof framework exhibits the critical effect of $L^{q}$-normalized building blocks and thus it can serve as an entry point for future growth. Our proof shows that the iterative pattern in current convex integration scheme need to be further improved to effectively leverage the threshold effect, see Section \ref{SEC002} for an in-depth explanation. The main result in this paper is now listed as follows.
\begin{theorem}\label{Main01}
There exists some $2<q\ll3$ such that for any given nonnegative smooth function $e(t):[0,T]\rightarrow[0,\infty)$, there is a weak solution $u\in C_{t}([0,T];L^{q}(\mathbb{T}^{3}))$ of the Navier-Stokes equation satisfying that $\int_{\mathbb{T}^{3}}|u(x,t)|^{2}dx=e(t)$ for all $t\in[0,T].$ 

\end{theorem}	

\begin{remark}
Combining Proposition \ref{prop06} and the fractional Gagliardo-Nirenberg inequalities established in \cite{BM2018}, we further obtain that the constructed solution $u$ belongs to $C_{t}W^{\alpha,q}$ for some $0<\alpha\ll1,$ which actually improves the result of $u\in C_{t}H^{\alpha}$ in \cite{BV2019}.
\end{remark}
\begin{remark}
From Theorem \ref{Main01}, we obtain the sharp non-uniqueness of weak solutions in $C_{t}L^{q}$ with zero initial data. For example, taking $e_{k}(t)=1-\cos kt$ for any integer $k\geq1$, we obtain a sequence of different nontrivial weak solutions $\{u_{k}\}_{k\geq1}\subset C_{t}L^{q}$ with $\|u_{k}(\cdot,t)\|_{L^{2}}^{2}=e_{k}(t)$ and $u_{k}(x,0)=0$ for $x\in\mathbb{T}^{3}$. Additionally, we see from the proof of Theorem \ref{Main01} that the value of $q$ is close to $2$ and remains uniform with respect to any kinetic energy $e(t)$.
\end{remark}

Before proving Theorem  \ref{Main01}, we first give a sketch for the concentration effect in intermittent convex integration scheme by analyzing the energy flux through each Littlewood-Paley shell, which can be considered as an analogy to the constructed perturbation with high concentration and oscillation in the following. Denote by $v_{m}:=\Delta_{m}u$ the Littlewood-Paley projection at frequency concentrated near $2^{m}$. For convenience, write $\lambda:=2^{m}$. By Bernstein's inequality and H\"{o}lder's inequality, we obtain that for any $1\leq p<q\leq\infty$,
\begin{align*}
\|v_{m}\|_{L^{p}}\sim\lambda^{3(\frac{1}{q}-\frac{1}{p})}\|v_{m}\|_{L^{q}}.
\end{align*}	
Multiplying \eqref{ZQ001} with $v_{m}$ and using integration by parts, we have
\begin{align*}
\frac{d}{dt}\|v_{m}\|_{L^{2}}^{2}+\nu\int_{\mathbb{T}^{3}}|\nabla v_{m}|^{2}dx=\int_{\mathbb{T}^{3}}v_{m}\cdot\Delta_{m}\mathrm{div}(u\otimes u)dx,
\end{align*}	
where
\begin{align*}
\text{Linear term }=\nu\int_{\mathbb{T}^{3}}|\nabla v_{m}|^{2}dx\sim\lambda^{2+6(\frac{1}{q}-\frac{1}{2})}\|v_{m}\|_{L^{q}}^{2},
\end{align*}	
and
\begin{align*}
\text{Nonlinear term }=&\int_{\mathbb{T}^{3}}v_{m}\cdot\Delta_{m}\mathrm{div}(u\otimes u)dx\notag\\
\lesssim&\|\nabla v_{m}\|_{L^{\infty}}\|v_{m}\|_{L^{2}}^{2}\lesssim\lambda^{1+\frac{3}{q}+6(\frac{1}{q}-\frac{1}{2})}\|v_{m}\|_{L^{q}}^{3}.
\end{align*}	
Here and below, the notations $a\lesssim b$ and $a \sim b$, respectively, represent that $a\leq Cb$ and $C^{-1}b\leq a\leq Cb$ for some universal positive constant $C$.   Therefore, if $\|v_{m}\|_{L^{q}}\rightarrow0$ as $m\rightarrow\infty$, we obtain that the linear term can be controlled by the nonlinear term when $q<3$, which actually corresponds to the antithesis of the Lady\v{z}enskaja-Prodi-Serrin criteria. Based on these facts, in the following we need to construct highly concentrated and oscillated building blocks in $L^{q}$ space as a central part of the perturbation in order to obtain the nonuniqueness. 

The rest of this paper is organized as follows. In Section \ref{SEC002}, we reduce the proof of Theorem  \ref{Main01} to the construction of an approximating sequence of smooth solutions to the Navier-Stokes-Reynolds system subject to some inductive estimates with superexponential decay, which is stated in Proposition \ref{pro01}. Section \ref{SEC003} is devoted to the proof of Proposition \ref{pro01} by manipulating the method of intermittent convex integration with the intermittent jets of $L^{q}$ normalization marked by both high oscillation and high concentration.

\section{Main iterative proposition}\label{SEC002}

For any integer $m\geq0$, starting from the trivial solution $(u_{0},\mathring{R}_{0})=(0,0)$, we aim to construct a converging sequence of smooth solutions $\{(u_{m},\mathring{R}_{m})\}$ of the following Navier-Stokes-Reynolds system
\begin{align}\label{ZQ002}
	\begin{cases}
		\partial_{t}u_{m}-\nu\Delta u_{m}+\mathrm{div}(u_{m}\otimes u_{m})+\nabla p_{m}=\mathrm{div}\mathring{R}_{m},\\
		\mathrm{div}u_{m}=0,
	\end{cases}\quad\mathrm{on}\;[0,T]\times\mathbb{T}^{3},	
\end{align}
where the pressure $p_{m}=\Delta^{-1}\mathrm{div}\mathrm{div}(\mathring{R}_{m}-u_{m}\otimes u_{m})$ with $\int_{\mathbb{T}^{3}}p_{m}=0$, the physical quantity  $\mathring{R}_{m}$, called the Reynolds stress, is a trace-free symmetric matrix. Remark that the Navier-Stokes-Reynolds system is frequently utilized in the context of computational fluid mechanics to numerically calculate the fluid. This is caused by the fact that if we carry out direct numerical simulation for the Navier-Stokes equation,  it needs to make use of sufficiently small regular mesh to precisely capture the sharp gradient changes for the fluids with high Reynolds number, which involves a large amount of computations such that this way is not practical. So in order to reduce computation complexity, it is effective to numerically analyze the macroscopically averaged velocity field $\bar{u}$ in larger regular mesh,  where $\bar{u}$ denotes the time or space average satisfying 
\begin{align}\label{F0101}
	\begin{cases}
		\partial_{t}\bar{u}-\nu\Delta \bar{u}+\mathrm{div}(\bar{u}\otimes \bar{u}+R)+\nabla \bar{p}=0,\\
		\mathrm{div}\bar{u}=0,\quad R=\overline{u\otimes u}-\bar{u}\otimes \bar{u}.
	\end{cases}
\end{align}
Here the Reynolds stress $R$ arises from the fact that the averaging operator cannot commute with the nonlinearity of $u\otimes u$. Remark that the Reynolds stress $R$ in \eqref{F0101} generally doesn't satisfies the trace-free condition. The purpose of adding this condition for $\mathring{R}_{m}$ in \eqref{ZQ002} is to meet the need of the energy iteration, see \eqref{E906} below for more details.  

Introduce the following parameters: for $m\geq0$,  $a,b>1$,  $\beta\in(0,1)$, $q\in(2,3)$, $A>q+2$,
\begin{align}\label{AAA01}
\varepsilon_{\ast}=\frac{3-q}{A}<\frac{1}{4},\quad\lambda_{m}=a^{b^{m}}, \quad \delta_{m}=\lambda_{m}^{-2\beta}\lambda_{1}^{2\beta+(q-2)(1+\varepsilon_{\ast})}\sup\limits_{t\in[0,T]}e(t).
\end{align}
The constructed sequence of solutions will satisfy the following inductive estimates: for $0<\varepsilon<\varepsilon_{\ast}/2,$
\begin{align}\label{QZ003}
\|\mathring{R}_{m}\|_{C_{t}L^{1}}\leq \vartheta_{m+1}^{\frac{q-2}{q}}\delta_{m+1}\lambda_{m}^{-2\varepsilon},\quad\vartheta_{m+1}=\lambda_{m+1}^{-q(1+\varepsilon_{\ast})},
\end{align}
and
\begin{align}\label{Z09}
\|u_{m}\|_{C_{t}L^{q}}\leq2\delta_{0}^{1/2}-\delta_{m}^{1/2},\quad\|\nabla u_{m}\|_{C_{t}L^{q}}\leq\lambda_{m}^{1+\varepsilon}\delta_{m}^{1/2},\quad\|u_{m}\|_{C^{1}_{x,t}}\leq\lambda_{m}^{4},	
\end{align}	
and, for all $t\in[0,T],$
\begin{align}\label{QZ005}
\begin{cases}
0\leq e(t)-\int_{\mathbb{T}^{3}}|u_{m}|^{2}dx\leq\vartheta_{m+1}^{\frac{q-2}{q}}\delta_{m+1},	\\
e(t)-\int_{\mathbb{T}^{3}}|u_{m}|^{2}dx\leq\vartheta_{m+1}^{\frac{q-2}{q}}\delta_{m+1}/10\;\Longrightarrow\;u_{m}(\cdot,t)\equiv0,\;\mathring{R}_{m}(\cdot,t)\equiv0.
\end{cases}
\end{align}	
In addition, we suppose that 
\begin{align}\label{AZ90}
a^{\frac{3-q-(q+2)\varepsilon_{\ast}}{3}}\in\mathbb{N},\quad b\in\mathbb{N}.
\end{align}
Remark that the condition in \eqref{AZ90} is assumed  to ensure that the $L^{q}$-normalized intermittent jets defined in Subsection \ref{IJ02} are smooth. It should be noted that the prerequisite for the condition of \eqref{AZ90} to hold is that $3-q-(q+2)\varepsilon_{\ast}>0$. This implies the critical phenomenon of exponent $q$ that the value of $q$ can arbitrarily approach the critical value $3$ only if $\varepsilon_{\ast}$ is accordingly chosen to make $0<\varepsilon_{\ast}<\frac{3-q}{q+2}$. So $\varepsilon_{\ast}$ can be regarded as an accuracy parameter for controlling the difference between the exponent $q$ and the threshold $3$. These facts show that the introduced building blocks of $L^{q}$ normalizaton in Subsection \ref{IJ02} contain the critical effect of exponent $q$, although the present scheme constrains the value of $q$ around $2$. Therefore, based on this framework, it needs to find some ways to further improve the iteration process of intermittent convex integration scheme such that the threshold effect can be fully exploited. In addition, in order to deepen the readers' understanding on the method of intermittent convex integration, we give a quantitative characterization for the ranges of the above introduced parameters in the process of the proof, which also offers a clear explanation of the relationship between these parameters. 

To prove Theorem \ref{Main01}, it suffices to establish the iterative proposition as follows.
\begin{prop}\label{pro01}
There exist some constants $2<q\ll3$, $0<2\varepsilon<\varepsilon_{\ast}<\frac{1}{4},$ $a,b\gg1$ and $0<\beta\ll1$ such that if $(u_{m},\mathring{R}_{m})$ satisfies the Navier-Stokes-Reynolds system in \eqref{ZQ002} under the condition of \eqref{QZ003}--\eqref{AZ90}, then there exists a second pair $(u_{m+1},\mathring{R}_{m+1})$ solving \eqref{ZQ002} and verifying \eqref{QZ003}--\eqref{AZ90} with $m+1$ substituting for $m$. Moreover, the velocity increment $u_{m+1}-u_{m}$ is nontrivial and satisfies 
\begin{align*}
\|u_{m+1}-u_{m}\|_{C_{t}L^{q}}\leq \kappa_{\ast}\delta^{1/2}_{m+1},	
\end{align*}
where $\kappa_{\ast}$ is a universal constant given by \eqref{C0015} below.	  
\end{prop}	
\begin{remark}
From the proof of Proposition \ref{pro01}, we know that the values of $q,\varepsilon_{\ast},\varepsilon,b,\beta,\kappa_{\ast}$ are all independent of $\sup_{t\in[0,T]} e(t)$ except for the parameter $a$. Moreover, if $\varepsilon\rightarrow\varepsilon_{\ast}/2$ or $\varepsilon\rightarrow0$, then $q\rightarrow2$, $b\rightarrow\infty$ and $\beta\rightarrow0$.
\end{remark}
A direct consequence of Proposition \ref{pro01} shows that Theorem \ref{Main01} holds.
\begin{proof}[Proof of Theorem \ref{pro01}]
From Proposition \ref{pro01}, we see that $\{u_{m}\}$ and $\{\mathring{R}_{m}\}$ are the Cauchy sequences in $C_{t}L^{q}$ and $C_{t}L^{1}$, respectively. This, together with the definition of weak solution, gives that Theorem \ref{Main01} holds.
\end{proof}

\section{ Implementation for intermittent convex integration scheme}\label{SEC003}
\subsection{Mollification}
Before proving Proposition \ref{pro01}, we first mollify $u_{m},\mathring{R}_{m}$ to be $u_{\ell},\mathring{R}_{\ell}$ with the aim of achieving high-order derivative estimates on these two mollified quantities in virtue of the mollification parameter $\ell$. To  be specific, let $\eta_{\ell}$ and $\xi_{\ell}$ be two standard mollifying kernels defined on $\mathbb{R}^{3}$ and $\mathbb{R}$, respectively, and denote
\begin{align*}
u_{\ell}=(u_{m}\ast_{x}\eta_{\ell})\ast_{t}\xi_{\ell},\quad\mathring{R}_{\ell}=(\mathring{R}_{m}\ast_{x}\eta_{\ell})\ast_{t}\xi_{\ell}+u_{\ell}\mathring{\otimes}u_{\ell}-((u_{m}\mathring{\otimes} u_{m})\ast_{x}\eta_{\ell})\ast_{t}\xi_{\ell}\big),	
\end{align*}	
where $(h\mathring{\otimes}h)_{ij}=h_{i}h_{j}-\frac{1}{3}\delta_{ij}|h|^{2}\mathrm{Id}$ with $\delta_{ij}$ and $\mathrm{Id}$ being the Kronecker symbol and the identity matrix, respectively. In order to ensure that $u_{m},\mathring{R}_{m}$ are mollified in the whole time domain $[0,T]$, before mollifications we actually perform the reflective extensions for $u_{m},\mathring{R}_{m}$ to make them be two continuous functions on $[-T,2T]$. That is, for any $x\in\mathbb{T}^{3}$, let
\begin{align*}
u_{m}(x,t)=
\begin{cases}
	u_{m}(x,-t),& \mathrm{if} \;-T\leq t\leq0,\\
	u_{m}(x,2T-t),& \mathrm{if} \;T\leq t\leq2T,
\end{cases}
\end{align*}
and
\begin{align*}
\mathring{R}_{m}(x,t)=
\begin{cases}
\mathring{R}_{m}(x,-t),& \mathrm{if} \;-T\leq t\leq0,\\
\mathring{R}_{m}(x,2T-t),& \mathrm{if} \;T\leq t\leq2T.
\end{cases}
\end{align*}	
It then follows from \eqref{ZQ002} that $(u_{\ell},\mathring{R}_{\ell})$ solves
\begin{align*}
	\begin{cases}
		\partial_{t}u_{\ell}-\nu\Delta u_{\ell}+\mathrm{div}(u_{\ell}\otimes u_{\ell})+\nabla p_{\ell}=\mathrm{div}\mathring{R}_{\ell},\\
		\mathrm{div}u_{\ell}=0,
	\end{cases}\quad\mathrm{on}\;[0,T]\times\mathbb{T}^{3},	
\end{align*}
where the new pressure $p_{\ell}$ is given by 
\begin{align*}
p_{\ell}=(p_{m}\ast_{x}\eta_{\ell})\ast_{t}\xi_{\ell}+\frac{1}{3}\big(|u_{\ell}|^{2}-(|u_{m}|^{2}\ast_{x}\eta_{\ell})\ast_{t}\xi_{\ell}\big).
\end{align*}	
In order to deal with the commutator $u_{\ell}\mathring{\otimes}u_{\ell}-((u_{m}\mathring{\otimes} u_{m})\ast_{x}\eta_{\ell})\ast_{t}\xi_{\ell}\big)$ in $\mathring{R}_{\ell}$, we need to establish the following commutator estimate.
\begin{lemma}\label{lem01}
For any multi-index $\alpha$ with $|\alpha|=k\in\mathbb{N}$, $1\leq p<\infty$, and $f,g\in C^{\infty}(\mathbb{T}^{3}\times[0,T])$, we have
\begin{align*}
&\|\nabla_{x,t}^{\alpha}[((f\ast_{x}\eta_{\ell})\ast_{t}\xi_{\ell})((g\ast_{x}\eta_{\ell})\ast_{t}\xi_{\ell})-((fg)\ast_{x}\eta_{\ell})\ast_{t}\xi_{\ell}]\|_{C_{t}L^{p}}\notag\\
&\lesssim\ell^{2-k}\|\nabla f\|_{C_{t}L^{2p}}\|\nabla g\|_{C_{t}L^{2p}}.
\end{align*}	
%Here and in the following, the notation $a\lesssim b$ represents that $a\leq Cb$ for some positive constant $C$ independent of $a,b$. 
\end{lemma}
Remark that the commutator-type estimate can be traced back to the work \cite{CWT1994}. Similar to Lemma 1 in \cite{CDS2012} and Proposition B.1 in \cite{L2019}, we give the proof of Lemma \ref{lem01} as follows.
\begin{proof}[Proof of Lemma \ref{lem01}]
For any multi-index $\alpha$ with $|\alpha|=k$, denote $\nabla_{x,t}^{\alpha}=\partial_{x}^{\alpha_{1}}\partial_{t}^{\alpha_{2}}=\partial^{\alpha_{1}}\partial^{\alpha_{2}}$ with $\alpha_{1}+\alpha_{2}=\alpha.$ From the product rule, we see
\begin{align*}
&\nabla_{x,t}^{\alpha}[((fg)\ast\eta_{\ell})\ast\xi_{\ell}-((f\ast\eta_{\ell})\ast\xi_{\ell})((g\ast\eta_{\ell})\ast\xi_{\ell})]\notag\\
&=((fg)\ast\partial^{\alpha_{1}}\eta_{\ell})\ast\partial^{\alpha_{2}}\xi_{\ell}-\sum_{\beta\leq\alpha}C_{\beta}^{\alpha}\nabla_{x,t}^{\beta}((f\ast\eta_{\ell})\ast\xi_{\ell})\nabla_{x,t}^{\alpha-\beta}((g\ast\eta_{\ell})\ast\xi_{\ell})\notag\\
&=((fg)\ast\partial^{\alpha_{1}}\eta_{\ell})\ast\partial^{\alpha_{2}}\xi_{\ell}-\nabla_{x,t}^{\alpha}((f\ast\eta_{\ell})\ast\xi_{\ell})-\nabla_{x,t}^{\alpha}((g\ast\eta_{\ell})\ast\xi_{\ell})\notag\\
&\quad-\sum_{0<\beta<\alpha}C_{\beta}^{\alpha}\nabla_{x,t}^{\beta}((f\ast\eta_{\ell})\ast\xi_{\ell})\nabla_{x,t}^{\alpha-\beta}((g\ast\eta_{\ell})\ast\xi_{\ell})\notag\\
&=(((f-f(x,t))(g-g(x,t)))\ast\partial^{\alpha_{1}}\eta_{\ell})\ast\partial^{\alpha_{2}}\xi_{\ell}\notag\\
&\quad-\sum_{0<\beta<\alpha}C_{\beta}^{\alpha}(((f-f(x,t))\ast\partial^{\beta_{1}}\eta_{\ell})\ast\partial^{\beta_{2}}\xi_{\ell})(((g-g(x,t))\ast\partial^{\alpha_{1}-\beta_{1}}\eta_{\ell})\ast\partial^{\alpha_{2}-\beta_{2}}\xi_{\ell}),
\end{align*}	
where we utilized the fact that for $\beta=\beta_{1}+\beta_{2}$,
\begin{align*}
&(f(x,t)\ast\partial^{\beta_{1}}\eta_{\ell})\ast\partial^{\beta_{2}}\xi_{\ell}=
\begin{cases}
f(x,t),&\text{if }|\beta|=0,\\
0,&\text{if }|\beta|>0.	
\end{cases}	
\end{align*}	
Denote
\begin{align*}
h(x,\cdot):=\int(f(x-y,\cdot)-f(x,t))(g(x-y,\cdot)-g(x,t))\partial^{\alpha_{1}}\eta_{\ell}(y)dy.
\end{align*}
Then we have
\begin{align*}
(((f-f(x,t))(g-g(x,t)))\ast\partial^{\alpha_{1}}\eta_{\ell})\ast\partial^{\alpha_{2}}\xi_{\ell}=h(x,\cdot)\ast\partial^{\alpha_{2}}\xi_{\ell}.
\end{align*}
Combining Minkowski's integral inequality, Young's convolution inequality, the integral form of mean value theorem and H\"{o}lder's inequality, we obtain 
\begin{align*}
&\|(((f-f(x,t))(g-g(x,t)))\ast\partial^{\alpha_{1}}\eta_{\ell})\ast\partial^{\alpha_{2}}\xi_{\ell}\|_{C_{t}L^{p}}\notag\\
&\leq\big\|\|h\|_{L^{p}}\ast|\partial^{\alpha_{2}}\xi_{\ell}|\big\|_{C_{t}}\leq \|h\|_{C_{t}L^{p}}\|\partial^{\alpha_{2}}\xi_{\ell}\|_{L^{1}_{t}}\notag\\
&\leq\|\partial^{\alpha_{2}}\xi_{\ell}\|_{L^{1}_{t}}\int\|f(\cdot-y)-f(\cdot)\|_{C_{t}L^{2p}}\|g(\cdot-y)-g(\cdot)\|_{C_{t}L^{2p}}|\partial^{\alpha_{1}}\eta_{\ell}(y)|dy\notag\\
&\leq\|\partial^{\alpha_{2}}\xi_{\ell}\|_{L^{1}_{t}}\|\nabla f\|_{C_{t}L^{2p}}\|\nabla g\|_{C_{t}L^{2p}}\int|y|^{2}|\partial^{\alpha_{1}}\eta_{\ell}(y)|dy\lesssim\ell^{2-k}\|\nabla f\|_{C_{t}L^{2p}}\|\nabla g\|_{C_{t}L^{2p}}.
\end{align*}	
By the same argument, we have
\begin{align*}
&\|(((f-f(x,t))\ast\partial^{\beta_{1}}\eta_{\ell})\ast\partial^{\beta_{2}}\xi_{\ell})(((g-g(x,t))\ast\partial^{\alpha_{1}-\beta_{1}}\eta_{\ell})\ast\partial^{\alpha_{2}-\beta_{2}}\xi_{\ell})\|_{C_{t}L^{p}}\notag\\
&\leq\|((f-f(x,t))\ast\partial^{\beta_{1}}\eta_{\ell})\ast\partial^{\beta_{2}}\xi_{\ell}\|_{C_{t}L^{2p}}\|((g-g(x,t))\ast\partial^{\alpha_{1}-\beta_{1}}\eta_{\ell})\ast\partial^{\alpha_{2}-\beta_{2}}\xi_{\ell}\|_{C_{t}L^{2p}}\notag\\
&\lesssim\ell^{2-k}\|\nabla f\|_{C_{t}L^{2p}}\|\nabla g\|_{C_{t}L^{2p}}.
\end{align*}
Combining these above inequalities, we complete the proof.
\end{proof}
 
Fix the value of $\ell$ as follows:
\begin{align}\label{ZQ007}
\lambda_{m}^{-6}\vartheta_{m+1}^{\frac{q-2}{q}}\ll\ell:=\frac{\vartheta_{m+1}^{\frac{q-2}{q}}\delta_{m+1}^{1/2}}{\lambda_{m}^{5}\delta^{1/2}_{m}}\ll\lambda_{m}^{-5}\vartheta_{m+1}^{\frac{q-2}{q}}.
\end{align}	
A direct application of Lemma \ref{lem01} gives the following estimates. 
\begin{lemma}\label{lem02}
For any $k\geq1$, we obtain
\begin{align*}
\|u_{\ell}-u_{m}\|_{C_{t}L^{q}}\lesssim\delta_{m+1}^{1/2}\lambda_{m}^{-4+\varepsilon}\vartheta_{m+1}^{\frac{q-2}{q}},\quad\|\nabla^{k}u_{\ell}\|_{C_{t}L^{q}}\lesssim\ell^{1-k}\lambda_{m}^{1+\varepsilon}\delta^{1/2}_{m},
\end{align*}
and
\begin{align*}
\|\mathring{R}_{\ell}\|_{C_{t}L^{1}}\leq2\vartheta_{m+1}^{\frac{q-2}{q}}\delta_{m+1}\lambda_{m}^{-2\varepsilon},\quad\|\nabla^{k}_{x,t}\mathring{R}_{\ell}\|_{C_{t}L^{1}}\lesssim\ell^{-k}\vartheta_{m+1}^{\frac{q-2}{q}}\delta_{m+1}\lambda_{m}^{-2\varepsilon}.
\end{align*}		
\end{lemma}
\begin{proof}
A combination of Minkowski's integral inequality, Young's convolution inequality, the integral form of mean value theorem,  \eqref{Z09} and \eqref{ZQ007} leads to that
\begin{align*}
\|u_{\ell}-u_{m}\|_{C_{t}L^{q}}\lesssim &\|u_{m}\ast\eta_{\ell}-u_{m}\|_{C_{t}L^{q}}\lesssim \ell\|\nabla u_{m}\|_{C_{t}L^{q}}\notag\\
\lesssim&\ell\lambda_{m}^{1+\varepsilon}\delta_{m}^{1/2}\lesssim\delta_{m+1}^{1/2}\lambda_{m}^{-4+\varepsilon}\vartheta_{m+1}^{\frac{q-2}{q}},
\end{align*}
and
\begin{align*}
\|\nabla^{k}u_{\ell}\|_{C_{t}L^{q}}\lesssim\ell^{1-k}\|\nabla u_{m}\|_{C_{t}L^{q}}\lesssim\ell^{1-k}\lambda_{m}^{1+\varepsilon}\delta_{m}^{1/2}.
\end{align*}		
Making use of \eqref{QZ003}, \eqref{ZQ007}, H\"{o}lder's inequality, Young's convolution inequality and Lemma \ref{lem01}, we deduce that
\begin{align*}
\|\mathring{R}_{\ell}\|_{C_{t}L^{1}}\leq&\|\mathring{R}_{m}\|_{C_{t}L^{1}}+C\ell^{2}\|\nabla u_{m}\|^{2}_{C_{t}L^{q}}\notag\\
\leq&\vartheta_{m+1}^{\frac{q-2}{q}}\delta_{m+1}\lambda_{m}^{-2\varepsilon}+C\ell^{2}\delta_{m}\lambda_{m}^{2+2\varepsilon} \leq 2\vartheta_{m+1}^{\frac{q-2}{q}}\delta_{m+1}\lambda_{m}^{-2\varepsilon},
\end{align*}
and
\begin{align*}
\|\nabla^{k}_{x,t}\mathring{R}_{\ell}\|_{C_{t}L^{1}}\lesssim&\ell^{-k}\|\mathring{R}_{m}\|_{C_{t}L^{1}}+\ell^{2-k}\|\nabla u_{m}\|^{2}_{C_{t}L^{q}}\notag\\
\lesssim&\ell^{-k}\vartheta_{m+1}^{\frac{q-2}{q}}\delta_{m+1}\lambda_{m}^{-2\varepsilon}+\ell^{2-k}\delta_{m}\lambda_{m}^{2+2\varepsilon} \lesssim \ell^{-k}\delta_{m+1}\lambda_{m}^{-2\varepsilon}\vartheta_{m+1}^{\frac{q-2}{q}}.
\end{align*}	

\end{proof}	
\begin{remark}
If we take $\ell=\frac{\vartheta_{m+1}^{\frac{q-2}{2q}}\delta_{m+1}^{1/2}}{\delta_{m}^{1/2}\lambda_{m}^{1+2\varepsilon}}$, we can derive the equivalence of upper bound estimates on the Reynolds term $(\mathring{R}_{\ell}\ast_{x}\eta_{\ell})\ast_{t}\xi_{\ell}$ and the commutator $u_{\ell}\mathring{\otimes}u_{\ell}-((u_{m}\mathring{\otimes} u_{m})\ast_{x}\eta_{\ell})\ast_{t}\xi_{\ell}\big)$. However, for later use we choose a smaller number $\frac{\vartheta_{m+1}^{\frac{q-2}{q}}\delta_{m+1}^{1/2}}{\lambda_{m}^{5}\delta^{1/2}_{m}}$ as its value in order to ensure the following fact: for $|t-t'|\leq2\ell$,
\begin{align}\label{W01}
\bigg|\int_{\mathbb{T}^{3}}|u_{m}(x,t)|^{2}-\int_{\mathbb{T}^{3}}|u_{m}(x,t')|^{2}\bigg|\lesssim \ell\|u_{m}\|_{C_{t}^{1}L^{\infty}}\|u_{m}\|_{C_{t}L^{1}}\ll\lambda_{m}^{-1}\vartheta_{m+1}^{\frac{q-2}{q}}.
\end{align}	 
  
\end{remark}

\subsection{$L^{q}$-normalized intermittent jets}\label{IJ02}
In the following, we plan to remake intermittent jets in \cite{BCV2022} rather than intermittent Beltrami waves in \cite{BV2019} as basic building blocks due to the advantage of having disjoint supports, which is similar to the property of Mikado flows in \cite{DS2017} and meanwhile avoids the loss of intermittent dimension. Now we begin with recalling the decomposition lemma on positive definite matrix (see \cite{DS2017,S2013}) as follows.
\begin{lemma}\label{lem05}
There exist two disjoint subsets $\Lambda_{\alpha}\subset \mathbb{S}^{2}\cap\mathbb{Q}^{3}$, $\alpha=0,1$, and smooth positive functions $\gamma_{(\zeta)}\in C^{\infty}(B_{\frac{1}{2}}(\mathrm{Id}))$, $\zeta\in\Lambda_{0}\cup\Lambda_{1}$ such that for any symmetric matrix $R\in B_{\frac{1}{2}}(\mathrm{Id})$,
$$R=\sum_{\zeta\in\Lambda_{\alpha}}\gamma_{(\zeta)}^{2}(R)\zeta\otimes\zeta,\quad\mathrm{for}\;\alpha=0,1.$$	
\end{lemma}
For any fixed $\zeta\in\Lambda_{0}\cup\Lambda_{1}$, denote by $A_{\zeta}\in\mathbb{S}^{2}\cap\mathbb{Q}^{3}$ a vector orthogonal to $\zeta$. Therefore, $\{\zeta,A_{\zeta},\zeta\times A_{\zeta}\}\subset\mathbb{S}^{2}\cap\mathbb{Q}^{3}$ forms an orthonormal basis in $\mathbb{R}^{3}$. Note that the number of elements in $\Lambda_{0}\cup\Lambda_{1}$ is finite. Then there exists a sufficiently large integer $N_{\Lambda}$ such that for any $\zeta\in\Lambda_{0}\cup\Lambda_{1}$,
\begin{align}\label{A01}
\{N_{\Lambda}\zeta,N_{\Lambda}A_{\zeta},N_{\Lambda}\zeta\times A_{\zeta}\}\subset N_{\Lambda}\mathbb{S}^{2}\cap\mathbb{N}^{3}.	
\end{align}	
For any $q\in(2,3)$, we choose a smooth function $\Phi:\mathbb{R}^{2}\rightarrow\mathbb{R}^{2}$ with support in the ball of radius $1$ satisfying that if $\phi=-\Delta\Phi,$ then 
\begin{align}\label{AQE01}
\int|\phi(x,y)|^{q}dxdy=1,\quad\dashint\phi^{2}(x,y)dxdy=c_{q},
\end{align}
for some positive constant $c_{q}>0$. Here and below, the notation $\dashint$ represents the integral average. Obviously there holds $\int\phi(x,y)dxdy=0$. Similarly, pick a smooth function $\psi:\mathbb{R}\rightarrow\mathbb{R}$ with support in a ball of radius $1$ such that
\begin{align}\label{AQE02}
\int\psi(z)dz=0,\quad\int|\psi(z)|^{q}dz=1,\quad \dashint\psi^{2}(z)dz=c_{q}^{\ast},
\end{align}
for some positive constant $c^{\ast}_{q}>0$. Define the following parameters:
\begin{align}\label{ZQ009}
	\sigma=\lambda_{m+1}^{-\frac{q+(q+2)\varepsilon_{\ast}}{3}},\quad r=\lambda_{m+1}^{-\frac{q-(4-q)\varepsilon_{\ast}}{3}},\quad\mu=\lambda_{m+1}^{1+2\varepsilon_{\ast}},	
\end{align}		
where $\varepsilon_{\ast}$ is given by \eqref{AAA01}. Remark that since $\lambda_{m+1}\gg1$ and $q\in(2,3)$, we have $\sigma\ll r\ll1$ and $\mu\gg1.$ From  \eqref{QZ003},  we see that $\vartheta_{m+1}=\sigma^{2}r$. We now concentrate the supports of $\phi,\Phi,\psi$ by defining the following rescaled functions:
\begin{align*}
\phi_{\sigma}(x,y)=\frac{\phi(x/\sigma,y/\sigma)}{\sigma^{2/q}},\quad	\Phi_{\sigma}(x,y)=\frac{\Phi(x/\sigma,y/\sigma)}{\sigma^{2/q}},\quad\psi_{r}(z)=\frac{\psi(z/r)}{r^{1/q}},
\end{align*}	
satisfying that $\phi_{\sigma}=-\sigma^{2}\Delta\Phi_{\sigma}$, and
\begin{align}\label{ZZW01}
\begin{cases}	
\dashint\phi_{\sigma}^{2}(x,y)dxdy=c_{q}\sigma^{\frac{2(q-2)}{q}},\quad\dashint\psi_{r}^{2}(z)dz=c_{q}^{\ast}r^{\frac{q-2}{q}},\\
\int|\phi_{\sigma}(x,y)|^{q}dxdy=\int|\psi_{r}(z)|^{q}dz=1,\\
\dashint\Phi_{\sigma}^{2}(x,y)dxdy\sim\sigma^{\frac{2(q-2)}{q}},\quad\int|\Phi_{\sigma}(x,y)|^{q}dxdy\sim1.
\end{cases}
\end{align}	
From \eqref{ZZW01}, we see that the rescaled functions $\phi_{\sigma},\Phi_{\sigma}, \psi_{r}$ are $L^{q}$-normalized. 

We now perform the periodic extensions for $\Phi_{\sigma},\phi_{\sigma}$ and $\psi_{r}$ such that they become functions defined on $\mathbb{T}^{2}$ and $\mathbb{T}$, respectively.  For simplicity, still denote the periodized functions as $\Phi_{\sigma},\phi_{\sigma},\psi_{r}$. Write $\lambda:=\lambda_{m+1}$. Define highly concentrated and oscillated functions $V_{(\zeta)}:\mathbb{T}^{3}\times\mathbb{R}\rightarrow\mathbb{R}$ by
\begin{align*}
&V_{(\zeta)}:=V_{\zeta,\sigma,r,\mu,\lambda}(x,t)\notag\\
&:=\frac{1}{(\lambda N_{\Lambda})^{2}}\psi_{r}(N_{\Lambda}\lambda\sigma(x\cdot\zeta+\mu t))\Phi_{\sigma}\big(N_{\Lambda}\lambda\sigma(x-\alpha_{\zeta})\cdot A_{\zeta},N_{\Lambda}\lambda\sigma(x-\alpha_{\zeta})\cdot(\zeta\times A_{\zeta})\big)\zeta.
\end{align*}	
Here $\alpha_{\zeta}\in\mathbb{R}^{3}$, $\zeta\in\Lambda_{0}\cup\Lambda_{1}$ represent the translations so that any two functions in $\{V_{(\zeta)}\}$ have disjoint supports. Now we introduce the $L^{q}$-normalized intermittent jets as follows:
\begin{align*}
&W_{(\zeta)}:=W_{\zeta,\sigma,r,\mu,\lambda}(x,t)\notag\\
&:=\psi_{r}(N_{\Lambda}\lambda\sigma(x\cdot\zeta+\mu t))\phi_{\sigma}\big(N_{\Lambda}\lambda\sigma(x-\alpha_{\zeta})\cdot A_{\zeta},N_{\Lambda}\lambda\sigma(x-\alpha_{\zeta})\cdot(\zeta\times A_{\zeta})\big)\zeta.
\end{align*}	
From \eqref{AZ90}, we obtain that $\lambda\sigma$ is an integer and therefore $V_{(\zeta)}$ and $W_{(\zeta)}$ are smooth building blocks. For simplicity of notations, denote
\begin{align*}
\psi_{(\zeta)}:=&\psi_{\zeta,\sigma,r,\mu,\lambda}:=\psi_{r}(N_{\Lambda}\lambda\sigma(x\cdot\zeta+\mu t)),\\
\Phi_{(\zeta)}:=&\Phi_{\zeta,\sigma,r,\lambda}:=\phi_{\sigma}\big(N_{\Lambda}\lambda\sigma(x-\alpha_{\zeta})\cdot A_{\zeta},N_{\Lambda}\lambda\sigma(x-\alpha_{\zeta})\cdot(\zeta\times A_{\zeta})\big)\zeta,\\
\phi_{(\zeta)}:=&\phi_{\zeta,\sigma,r,\lambda}:=\phi_{\sigma}\big(N_{\Lambda}\lambda\sigma(x-\alpha_{\zeta})\cdot A_{\zeta},N_{\Lambda}\lambda\sigma(x-\alpha_{\zeta})\cdot(\zeta\times A_{\zeta})\big)\zeta.
\end{align*}	
In view of \eqref{A01} and \eqref{ZZW01}, we know that $W_{(\zeta)}$ is $(\mathbb{T}/N_{\Lambda}\lambda\sigma)^{3}$-periodic, 
\begin{align*}
\int W_{(\zeta)}dx=0,\quad W_{(\zeta)}\otimes W_{(\zeta')}\equiv 0,\;\,\text{if }\zeta\neq\zeta'\in \Lambda_{0}\cup\Lambda_{1},
\end{align*}	
and 
\begin{align*}
\int|W_{(\zeta)}|^{q}dx=1,\quad\dashint W_{(\zeta)}\otimes W_{(\zeta)}dx=c_{q}c_{q}^{\ast}(\sigma^{2}r)^{\frac{q-2}{q}}\zeta\otimes\zeta.
\end{align*}	
This, in combination with Lemma \ref{lem05}, shows that for any symmetric matrix $R\in B_{\frac{1}{2}}(\mathrm{Id})$,
\begin{align}\label{ZQ10}
\sum\limits_{\zeta\in\Lambda_{\alpha}}\gamma_{(\zeta)}^{2}(R)	\dashint W_{(\zeta)}\otimes W_{(\zeta)}dx=c_{q}c_{q}^{\ast}(\sigma^{2}r)^{\frac{q-2}{q}}R,\quad\alpha=0,1.
\end{align}	
From these above facts, we see that these intermittent jets $W_{(\zeta)}$, $\zeta\in\Lambda_{0}\cup\Lambda_{1}$ belong to $L^{q}$-normalized building blocks with high concentration and oscillation. According to the roles of $\lambda,\sigma,r,\mu$ played in the above constructions,  $r$ is a concentration parameter, $\lambda,\mu$ are two oscillation parameters, while the parameter $\sigma$ plays both roles in concentration and oscillation. It is worth emphasizing that the single concentration only generates sparsely distributed high-amplitude functions with small support. Building on this, we further add high oscillation to produce highly densely distributed high-amplitude functions with smaller support, whose graphs resemble densely distributed Dirac delta functions. These constructions are actually rooted in the fact that smooth solutions are unique and conversely we need to construct a converging sequence of solutions with these sharp features so as to make their limit be in some weak Lebesgue spaces. In addition, by a straightforward computation, these $L^{q}$-normalized building blocks satisfy that for $1\leq p\leq\infty$, $j,k\geq0$,
\begin{align}\label{M01}
\begin{cases}	
\|\partial_{t}^{j}\nabla^{k}\psi_{(\zeta)}\|_{L^{p}}\lesssim r^{1/p-1/q}\big(\frac{\lambda\sigma\mu}{r}\big)^{j}\big(\frac{\lambda\sigma}{r}\big)^{k},\\
\|\nabla^{k}\phi_{(\zeta)}\|_{L^{p}}+\|\nabla^{k}\Phi_{(\zeta)}\|_{L^{p}}\lesssim\sigma^{2(1/p-1/q)}\lambda^{k},\\
\|\partial_{t}^{j}\nabla^{k}W_{(\zeta)}\|_{L^{p}}+\lambda^{2}\|\partial_{t}^{j}\nabla^{k}V_{(\zeta)}\|_{L^{p}}\lesssim(\sigma^{2}r)^{1/p-1/q}\lambda^{k}\big(\frac{\lambda\sigma\mu}{r}\big)^{j},
\end{cases}	
\end{align}	
where we also utilized the fact of $\sigma\ll r.$

\subsection{Construction on the perturbation}
To begin with, we see that the perturbation $w_{m+1}:=u_{m+1}-u_{\ell}$ should satisfy
\begin{align*}
	\mathrm{div}\mathring{R}_{m+1}-\nabla(p_{m+1}-p_{\ell})=&\partial_{t}w_{m+1}+\mathrm{div}(w_{m+1}\otimes w_{m+1}+\mathring{R}_{\ell})-\nu\Delta w_{m+1}\notag\\
	&+\mathrm{div}(u_{\ell}\otimes w_{m+1}+w_{m+1}\otimes u_{\ell}).
\end{align*}	
From this equation, we see that the problem of constructing the new solution $(u_{m+1},\mathring{R}_{m+1})$ is reduced to construction of the perturbation $w_{m+1}$, whose core lies in making the constructed tensor $w_{m+1}\otimes w_{m+1}$ cancel the previous Reynolds stress $\mathring{R}_{\ell}$. Observe that $\mathring{R}_{\ell}$ is a symmetric matrix. So it is natural to consider whether we could construct a special $w_{m+1}$ such that $w_{m+1}\otimes w_{m+1}=\rho(\mathrm{Id}-\mathring{R}_{\ell}/\rho)$ with the divergence of $\rho\mathrm{Id}$ being put into the pressure. To further simplify the problem, we need to find suitable function $\rho$ so that $\mathrm{Id}-\mathring{R}_{\ell}/\rho$ is a positive definite matrix and thus can be decomposed into a combination of explicit basis matrices with positive coefficients in virtue of Lemma \ref{lem05}. This in turn determines the exact form of the perturbation $w_{m+1}$. Following this rough idea, we carry out strict construction on the perturbation step by step in the following. 

Introduce a smooth function $\hat{\chi}:\mathbb{R}^{3\times3}\rightarrow[0,1]$ satisfying that
\begin{align*}
\hat{\chi}(y)=
\begin{cases}
1,&\text{if }0\leq|y|\leq\frac{3}{4},\\
0,&\text{if }|y|\geq1.	
\end{cases}		
\end{align*}	
For $i\geq0$, define $\hat{\chi}_{i}(y)=\hat{\chi}(4^{-i}y)$ and
\begin{align*}
\tilde{\chi}_{i}(y)=
\begin{cases}
(\hat{\chi}_{i}-\hat{\chi}_{i-1})^{\frac{1}{2}},&\text{if }i\geq1,\\
\hat{\chi}_{0}^{\frac{1}{2}},&\text{if }i=0,	
\end{cases}		
\end{align*}	
which implies that $\sum_{i\geq0}\tilde{\chi}^{2}_{i}\equiv1$ on $\mathbb{R}^{3\times3}$.  Using this partition of unity, we split the Reynolds stress as follows: for any $i\geq0$,
\begin{align*}
\chi_{(i)}(x,t):=\chi_{i,m+1}(x,t)=\tilde{\chi}_{i}\bigg(\frac{\mathring{R}_{\ell}(x,t)}{\vartheta_{m+1}^{\frac{q-2}{q}}\delta_{m+1}\lambda_{m}^{-2\varepsilon}}\bigg).	
\end{align*}	
Denote $\rho_{i}:=4^{i+1}\vartheta_{m+1}^{\frac{q-2}{q}}\delta_{m+1}\lambda_{m}^{-2\varepsilon}$ for $i\geq1$, and $\rho_{0}:=((\rho^{1/2})\ast_{t}\xi_{\ell})^{2}$, where
\begin{align*}
\rho(t):=\frac{\max\{\tilde{e}(t)-2^{-1}\vartheta_{m+2}^{\frac{q-2}{q}}\delta_{m+2},0\}}{3\int_{\mathbb{T}^{3}}\chi^{2}_{(0)}dx},
\end{align*}
and
\begin{align*}
\tilde{e}(t):=e(t)-\int_{\mathbb{T}^{3}}|u_{m}|^{2}dx-3\sum_{i\geq1}\rho_{i}\int_{\mathbb{T}^{3}}\chi^{2}_{(i)}dx.
\end{align*}	
Using the Chebyshev's inequality and Lemma \ref{lem02}, we have
\begin{align*}
|\{x:|\mathring{R}_{\ell}|\geq 4|\mathbb{T}^{3}|^{-1}\vartheta_{m+1}^{\frac{q-2}{q}}\delta_{m+1}\lambda_{m}^{-2\varepsilon}\}|\leq\frac{|\mathbb{T}^{3}|\|\mathring{R}_{\ell}\|_{L^{1}}}{4\vartheta_{m+1}^{\frac{q-2}{q}}\delta_{m+1}\lambda_{m}^{-2\varepsilon}}	\leq\frac{|\mathbb{T}^{3}|}{2},
\end{align*}	
which indicates that $\int_{\mathbb{T}^{3}}\chi_{(0)}^{2}dx\geq\frac{|\mathbb{T}^{3}|}{2}$. Then it follows from Young's convolution inequality that
\begin{align}\label{Q08}
\|\rho_{0}\|_{L^{\infty}}\leq&\|\rho\|_{L^{\infty}}\leq\frac{\|e(\cdot)-\int_{\mathbb{T}^{3}}|u_{m}|^{2}dx\|_{L^{\infty}}}{3\int_{\mathbb{T}^{3}}\chi_{(0)}^{2}dx}\leq\frac{2\vartheta_{m+1}^{\frac{q-2}{q}}\delta_{m+1}}{3|\mathbb{T}^{3}|}\leq\vartheta_{m+1}^{\frac{q-2}{q}}\delta_{m+1},
\end{align}	
and
\begin{align}\label{D009}
\|\rho_{0}\|_{C^{k}_{t}}\lesssim\vartheta_{m+1}^{\frac{q-2}{q}}\delta_{m+1}\ell^{-k},\quad\|\rho_{0}^{1/2}\|_{C_{t}^{k}}\lesssim\vartheta_{m+1}^{\frac{q-2}{2q}}\delta^{1/2}_{m+1}\ell^{-k},\quad\text{for }k\geq1.	
\end{align}	
Denote by $i_{\max}$ the smallest integer such that
\begin{align*} 
\|\mathring{R}_{\ell}\|_{L^{\infty}}\leq C\|\mathring{R}_{\ell}\|_{W^{4,1}}\leq C\ell^{-4}\|\mathring{R}_{\ell}\|_{L^{1}}\leq C\ell^{-4}\vartheta_{m+1}^{\frac{q-2}{q}}\delta_{m+1}\lambda_{m}^{-2\varepsilon}\leq 3\rho_{i_{\max}}/16,
\end{align*}
which implies that $\chi_{(i)}\equiv0$ on $\mathbb{R}^{3\times3}$ for $i>i_{\max}\sim\ln\lambda_{m+1}$. For later use, we have from \eqref{QZ003} and \eqref{ZQ007}  that $\ell^{-4}=\lambda_{m+1}^{4[(q-2)(1+\varepsilon_{\ast})+\frac{5+\beta(b-1)}{b}]}$, and then
\begin{align}\label{Z90}
2^{i_{\max}}\sim\lambda_{m+1}^{2[(q-2)(1+\varepsilon_{\ast})+\frac{5+\beta(b-1)}{b}]}.	
\end{align}	

Note that
\begin{align}\label{K01}
\frac{3}{64}\leq\bigg|\frac{\mathring{R}_{\ell}}{\rho_{i}}\bigg|\leq\frac{1}{4},\quad\text{for }x\in\mathrm{supp}\chi_{(i)}, \;i\geq1,
\end{align}	
which yields that $\mathrm{Id}-\frac{\mathring{R}_{\ell}}{\rho_{i}}\in B_{\frac{1}{2}}(\mathrm{Id})$. This, together with the Chebyshev's inequality, leads to that if $a\gg1$,
\begin{align}\label{ZQ015}
3\sum_{i\geq1}\rho_{i}\int_{\mathbb{T}^{3}}\chi^{2}_{(i)}dx\leq&3\sum^{i_{\max}}_{i\geq1}\rho_{i}|\mathrm{supp}\chi_{(i)}|\leq3\sum^{i_{\max}}_{i\geq1}\rho_{i}\bigg|\bigg\{x:\bigg|\frac{\mathring{R}_{\ell}}{\rho_{i}}\bigg|\geq\frac{3}{64}\bigg\}\bigg|\notag\\
\lesssim&i_{\max}\|\mathring{R}_{\ell}\|_{C_{t}L^{1}}\leq \vartheta_{m+1}^{\frac{q-2}{q}}\delta_{m+1}\lambda_{m}^{-\varepsilon}.
\end{align}	
By a slight abuse of notation, still write
\begin{align}\label{E81}
\frac{\mathring{R}_{\ell}}{\rho_{0}(t)}=
\begin{cases}
\frac{\mathring{R}_{\ell}}{\rho_{0}(t)},&\text{if }\chi_{(0)}\neq0\;\text{and }\mathring{R}_{\ell}\neq0,\\
0,& \text{otherwise}.
\end{cases}		
\end{align}	
Remark that if $\chi_{(0)}\neq0\;\text{and }\mathring{R}_{\ell}\neq0$, we see from \eqref{QZ005} and \eqref{ZQ015} that $\rho_{0}>0$. To demonstrate $\|\mathring{R}_{\ell}/\rho_{0}\|_{L^{\infty}(\mathrm{supp}\chi_{(0)})}\leq1/4$ and thus $\mathrm{Id}-\frac{\mathring{R}_{\ell}}{\rho_{0}}\in B_{\frac{1}{2}}(\mathrm{Id})$, it suffices to prove that in the case when $\chi_{(0)}\neq0\;\text{and }\mathring{R}_{\ell}\neq0$,
\begin{align}\label{E86}
\rho_{0}\geq 16\vartheta_{m+1}^{\frac{q-2}{q}}\delta_{m+1}\lambda_{m}^{-2\varepsilon},\quad\text{on supp}\chi_{(0)}.
\end{align}
First, we have from Lemma \ref{lem02} that for any $|t_{1}-t_{2}|<\ell$,
\begin{align}\label{W08}
&\bigg|\sum_{i\geq1}\rho_{i}\int_{\mathbb{T}^{3}}\big(\chi_{(i)}^{2}(x,t_{1})-\chi_{(i)}^{2}(x,t_{2})\big)dx\bigg|\leq2\sum_{i\geq1}\rho_{i}\ell\|\partial_{t}\chi_{(i)}\|_{C_{t}L^{1}}\notag\\
&\lesssim\sum_{i\geq1}\ell\|\partial_{t}\mathring{R}_{\ell}\|_{C_{t}L^{1}}\lesssim\vartheta_{m+1}^{\frac{q-2}{q}}\delta_{m+1}\lambda_{m}^{-\varepsilon}.
\end{align}	
This, in combination with \eqref{QZ005}, \eqref{W01} and \eqref{ZQ015}, shows that if $0<b\beta<1/2$ and $a\gg1$,
\begin{align}\label{W902}
	&|\rho_{0}-\rho|=|(\rho^{1/2}\ast\xi_{\ell}+\rho^{1/2})(\rho^{1/2}\ast\xi_{\ell}-\rho^{1/2})|\notag\\
	&\leq2\|\rho^{1/2}\|_{L^{\infty}_{t}}\int_{0}^{T}|\rho^{1/2}(t-s)-\rho^{1/2}(t)|\xi_{\ell}(s)ds\notag\\
&\leq2\|\rho^{1/2}\|_{L^{\infty}_{t}}\int_{0}^{T}|\rho(t-s)-\rho(t)|^{1/2}\xi_{\ell}(s)ds\lesssim\frac{\vartheta_{m+1}^{\frac{q-2}{q}}\delta_{m+1}^{1/2}}{\lambda_{m}^{1/2}}\leq\frac{\vartheta_{m+1}^{\frac{q-2}{q}}\delta_{m+1}}{100|\mathbb{T}^{3}|},
\end{align}
which reads that
\begin{align*}
\rho_{0}\geq&\rho-|\rho_{0}-\rho|\geq \frac{\tilde{e}(t)-\frac{\vartheta_{m+2}^{\frac{q-2}{q}}\delta_{m+2}}{2}}{3\int_{\mathbb{T}^{3}}\chi^{2}_{(0)}dx}-\frac{\vartheta_{m+1}^{\frac{q-2}{q}}\delta_{m+1}}{100|\mathbb{T}^{3}|}\notag\\
\geq&\frac{\vartheta_{m+1}^{\frac{q-2}{q}}\delta_{m+1}}{50|\mathbb{T}^{3}|}\geq16\vartheta_{m+1}^{\frac{q-2}{q}}\delta_{m+1}\lambda_{m}^{-2\varepsilon}.
\end{align*}	
Therefore, \eqref{E86} holds. These facts actually reflect the reason why the second condition in \eqref{QZ005} is imposed, which essentially stems from the handling for the difficulty induced by potential zero points of $e(t)$.

For $i\geq0,$ we introduce the coefficient function as follows:
\begin{align}\label{ZQ29}
a_{(\zeta)}:=a_{\zeta,i,m+1}:=(c_{q}c_{q}^{\ast})^{-1/2}\vartheta_{m+1}^{-\frac{q-2}{2q}}\rho_{i}^{1/2}\chi_{(i)}\gamma_{(\zeta)}\bigg(\mathrm{Id}-\frac{\mathring{R}_{\ell}}{\rho_{i}}\bigg),
\end{align}	
where $\vartheta_{m+1}$ and $c_{q},c_{q}^{\ast}$ are, respectively, given by \eqref{QZ003} and \eqref{AQE01}--\eqref{AQE02}. Then we construct the perturbation $w_{m+1}:=w_{m+1}^{(p)}+w_{m+1}^{(c)}+w_{m+1}^{(t)}$, where 
\begin{align*}
w_{m+1}^{(p)}:=&\sum\limits_{i\geq0}\sum\limits_{\zeta\in\Lambda_{(i)}}a_{(\zeta)}W_{(\zeta)},\\
w_{m+1}^{(c)}:=&\sum\limits_{i\geq0}\sum\limits_{\zeta\in\Lambda_{(i)}}\Big(\mathrm{curl}(\nabla a_{(\zeta)}\times V_{(\zeta)})+\frac{1}{(N_{\Lambda}\lambda_{m+1})^{2}}\nabla(a_{(\zeta)}V_{(\zeta)})\times \mathrm{curl}(\Phi_{(\zeta)}\zeta)\Big),\\
w_{m+1}^{(t)}:=&-\frac{1}{\mu}\sum\limits_{i\geq0}\sum\limits_{\zeta\in\Lambda_{(i)}}\mathbb{P}_{H}\mathbb{P}_{\neq0}\big(a_{(\zeta)}^{2}\phi_{(\zeta)}^{2}\psi_{(\zeta)}^{2}\zeta\big).
\end{align*}	
Here $\Lambda_{(i)}:=\Lambda_{i\,\mathrm{mod}\,2}$. This notation is introduced to ensure that there has no cross term in the tensor product of $w_{m+1}^{(p)}\otimes w_{m+1}^{(p)}$. In fact, for $i,j\geq0$, $\zeta\in\Lambda_{(i)},\zeta'\in\Lambda_{(j)}$,  if $|i-j|=1$, we have $\zeta\neq\zeta'$, and if $|i-j|\geq2$, then $\chi_{(i)}\chi_{(j)}=0$. In either case, there all has $a_{(\zeta)}W_{(\zeta)}\otimes a_{(\zeta')}W_{(\zeta')}=0$.  Remark that $w_{m+1}^{(p)}$ is the principal part used to cancel the Reynolds stress $\mathring{R}_{\ell}$ such that
\begin{align*}
\mathrm{div}(w_{m+1}^{(p)}\otimes w_{m+1}^{(p)}+\mathring{R}_{\ell})\sim&	\frac{1}{\mu}\sum\limits_{i\geq0}\sum\limits_{\zeta\in\Lambda_{(i)}}\partial_{t}\mathbb{P}_{H}\mathbb{P}_{\neq0}\big(a_{(\zeta)}^{2}\phi_{(\zeta)}^{2}\psi_{(\zeta)}^{2}\zeta\big)\notag\\
&+\text{(pressure gradient)+(high frequence error)}.
\end{align*}	
The time corrector $w_{m+1}^{(t)}$ is then added to achieve
\begin{align*}
\partial_{t}w_{m+1}^{(t)}+\mathrm{div}(w_{m+1}^{(p)}\otimes w_{m+1}^{(p)}+\mathring{R}_{\ell})\sim\text{(pressure gradient)+(high frequence error)}.
\end{align*}	
Finally, the incompressibility corrector $w_{m+1}^{(c)}$ is introduced to make 
$$\mathrm{div}(w_{m+1}^{(p)}+w_{m+1}^{(c)})=\sum\limits_{i\geq0}\sum\limits_{\zeta\in\Lambda_{(i)}}\mathrm{div}\,\mathrm{curl}\,\mathrm{curl}(a_{(\zeta)}V_{(\zeta)})=0.$$

\subsection{Estimates on the perturbation}
We start from the estimates on the coefficient function $a_{(\zeta)}$ as follows.
\begin{lemma}\label{lem06}
For $1\leq p\leq\infty$, we have
\begin{align*}
\begin{cases}
\|a_{(\zeta)}\|_{L^{p}}\leq\sqrt[p]{11}(c_{q}c_{q}^{\ast})^{-1/2}\|\gamma_{(\zeta)}\|_{L^{\infty}}\delta_{m+1}^{1/2},&\text{if }i=0,\\
\|a_{(\zeta)}\|_{L^{p}}\lesssim2^{\frac{(p-2)i}{p}}\delta_{m+1}^{1/2}\lambda_{m}^{-\varepsilon},&\text{if } i\geq1,
\end{cases}
\end{align*}
and, for all $0\leq i\leq i_{\max}$ and $k\in\mathbb{N}$, 
\begin{align*}
\quad\|\nabla^{k}_{x,t}a_{(\zeta)}\|_{L^{\infty}}\lesssim\delta_{m+1}^{1/2}\ell^{-k-8},\quad\forall\,0\leq i\leq i_{\max}.
\end{align*}		

\end{lemma}

\begin{proof}
To begin with, we deduce from Lemma \ref{lem02}, \eqref{Q08}--\eqref{K01} and the Chebyshev's inequality that for any $0\leq i\leq i_{\max}$,
\begin{align}\label{DF011}
	|\mathrm{supp}\chi_{(i)}|\leq&|\{x:|\mathring{R}_{\ell}|\geq 3\cdot4^{i-2}\vartheta_{m+1}^{\frac{q-2}{q}}\delta_{m+1}\lambda_{m}^{-2\varepsilon}\}|\notag\\
	\leq&\frac{\|\mathring{R}_{\ell}\|_{L^{1}}}{3\cdot4^{i-2}\vartheta_{m+1}^{\frac{q-2}{q}}\delta_{m+1}\lambda_{m}^{-2\varepsilon}}\leq\frac{11}{4^{i}}.
\end{align}		
In view of \eqref{ZQ29}, we have from \eqref{Q08} and \eqref{DF011} that if $i=0$, 
\begin{align*}
 \|a_{(\zeta)}\|_{L^{p}}\leq&(c_{q}c_{q}^{\ast})^{-1/2}\delta_{m+1}^{1/2}\|\gamma_{(\zeta)}\|_{L^{\infty}}|\mathrm{supp}\chi_{(0)}|^{1/p}\notag\\
 \leq&\sqrt[p]{11}(c_{q}c_{q}^{\ast})^{-1/2}\|\gamma_{(\zeta)}\|_{L^{\infty}}\delta_{m+1}^{1/2},
\end{align*}
and if $i\geq1$, 
\begin{align*}
\|a_{(\zeta)}\|_{L^{p}}\lesssim2^{i+1}\delta_{m+1}^{1/2}\lambda_{m}^{-\varepsilon}|\mathrm{supp}\chi_{(i)}|^{1/p}\lesssim2^{\frac{(p-2)i}{p}}\delta_{m+1}^{1/2}\lambda_{m}^{-\varepsilon}.
\end{align*}	
Utilizing Proposition 4.1 in \cite{DS2014}, Sobolev embedding theorem, \eqref{Q08}--\eqref{K01}, \eqref{E81}--\eqref{ZQ29} and Lemma \ref{lem02}, we have
\begin{align*}
\|\nabla^{k}_{x,t}\chi_{(i)}\|_{L^{\infty}(\mathrm{supp}\chi_{(i)})}\lesssim&\|\nabla^{k}_{x,t}(\mathring{R}_{\ell}/\rho_{i})\|_{L^{\infty}(\mathrm{supp}\chi_{(i)})}\sum_{j\leq k}\|\mathring{R}_{\ell}/\rho_{i}\|_{L^{\infty}(\mathrm{supp}\chi_{(i)})}^{j-1}\notag\\
\lesssim&\|\nabla_{x,t}^{k}(\mathring{R}_{\ell}/\rho_{i})\|_{C_{t}W^{4,1}}\lesssim 4^{-i}\ell^{-k-4},
\end{align*}
and
\begin{align*}
\|\nabla_{x,t}^{k}\gamma_{(\zeta)}\|_{L^{\infty}(\mathrm{supp}\chi_{(i)})}\lesssim&\|\nabla_{x,t}^{k}(\mathring{R}_{\ell}/\rho_{i})\|_{L^{\infty}(\mathrm{supp}\chi_{(i)})}\sum_{j\leq k}\|\mathrm{Id}-\mathring{R}_{\ell}/\rho_{i}\|_{L^{\infty}(\mathrm{supp}\chi_{(i)})}^{j-1}\notag\\
\lesssim&\|\nabla_{x,t}^{k}(\mathring{R}_{\ell}/\rho_{i})\|_{C_{t}W^{4,1}}\lesssim 4^{-i}\ell^{-k-4}.
\end{align*}
It then follows that for $i\geq0$,
\begin{align*}
|\nabla_{x,t}^{k} a_{(\zeta)}|\lesssim&2^{i+1}\delta_{m+1}^{1/2}\sum_{0\leq j\leq k}|\nabla^{j}_{x,t}\chi_{(i)}|\bigg|\nabla^{k-j}_{x,t}\gamma_{(\zeta)}\Big(\mathrm{Id}-\frac{\mathring{R}_{\ell}}{\rho_{i}}\Big)\bigg|\lesssim\delta_{m+1}^{1/2}\ell^{-k-8}.
\end{align*}

\end{proof}

Define the new velocity field $u_{m+1}$ as
\begin{align*}
u_{m+1}:=w_{m+1}+u_{\ell}.	
\end{align*}
For simplicity, denote
\begin{align}\label{C0015}
	\kappa_{\ast}:=6N_{0}\sqrt[p]{11}(c_{q}c_{q}^{\ast})^{-1/2}\max_{\zeta\in\Lambda_{(0)}}\|\gamma_{(\zeta)}\|_{L^{\infty}},
\end{align}	
where $N_{0}$ is the number of elements in $\Lambda_{(0)}$. 
\begin{prop}\label{prop06}
There exists a constant $2<q\ll3$ such that
\begin{align*}
\|w_{m+1}\|_{L^{q}}\leq \kappa_{\ast}\delta_{m+1}^{1/2},\quad	\|u_{m+1}-u_{m}\|_{L^{q}}
\leq&\kappa_{\ast}\delta_{m+1}^{1/2},
\end{align*}
and, for $k\geq1$, 
\begin{align*}
\|w_{m+1}\|_{W^{k,q}} \leq\frac{1}{2}\delta_{m+1}^{1/2}\lambda_{m+1}^{k+\varepsilon},\quad\|w_{m+1}\|_{C_{x,t}^{1}}\leq\frac{1}{2}\lambda_{m+1}^{4},
\end{align*}	
where $\kappa_{\ast}$ is given by \eqref{C0015}.
\end{prop}

Before giving the proof of Proposition \ref{prop06}, we first state the improved H\"{o}lder's inequality established in \cite{MS2018}, which was inspired by Lemma 3.7 in \cite{BV2019}. 
\begin{lemma}\label{lem26}
Set $p\in[1,\infty]$ and $n\geq1$. Let $f,g:\mathbb{T}^{n}\rightarrow\mathbb{R}$ be two smooth functions. Then for any $\lambda\in\mathbb{N}$,
\begin{align*}
\|f(\lambda\cdot)g\|_{L^{p}}\leq \|f\|_{L^{p}}\|g\|_{L^{p}}+C_{p}\lambda^{-1/p}\|f\|_{L^{p}}\|g\|_{C^{1}}.	
\end{align*}	
%In particular, for $|\lambda|\gg1$, it follows that $\|f(\lambda\cdot)g\|_{L^{p}}\lesssim \|f\|_{L^{p}}\|g\|_{L^{p}}$.
\end{lemma}	

\begin{proof}[Proof of Proposition \ref{prop06}]
For the purpose of letting
\begin{align}\label{M012}
 \ell^{-10}(\lambda_{m+1}\sigma)^{-1/q}=\lambda_{m+1}^{10[(q-2)(1+\varepsilon_{\ast})+\frac{5+\beta(b-1)}{b}]-\frac{3-q-(q+2)\varepsilon_{\ast}}{3q}}\ll1,
\end{align}
we need $10[(q-2)(1+\varepsilon_{\ast})+\frac{5+\beta(b-1)}{b}]-\frac{3-q-(q+2)\varepsilon_{\ast}}{3q}<0$. By taking $b> \frac{60q[5+\beta(b-1)]}{3-q-(q+2)\varepsilon_{\ast}}$, we reduce it to $10[(q-2)(1+\varepsilon_{\ast})+\frac{5+\beta(b-1)}{b}]\leq\frac{3-q-(q+2)\varepsilon_{\ast}}{6q}$. Remark that using \eqref{AAA01}, we have  $3-q-(q+2)\varepsilon_{\ast}>0$. So it suffices to require that
\begin{align}\label{I01}
	2<q\leq2+\frac{-\tau_{1}+\sqrt{\tau_{1}^{2}+240(1+\varepsilon_{\ast})\tau_{2}}}{120(1+\varepsilon_{\ast})},\quad b> \frac{60q(5+\beta(b-1))}{3-q-(q+2)\varepsilon_{\ast}},
\end{align}
where 
\begin{align*}
\tau_{1}:=121(1+\varepsilon_{\ast})+\frac{60(5+\beta(b-1))}{b},\quad\tau_{2}:=1-\frac{120(5+\beta(b-1))}{b}.	
\end{align*}		
Making use of \eqref{Z90} and the fact that $a\gg1$, we have $i_{\max}\ell\ll1$. Moreover, in order to make
\begin{align}\label{M013}
	2^{\frac{(q-2)i_{\max}}{q}}\lambda_{m}^{-\varepsilon}\lesssim\lambda_{m+1}^{\frac{2(q-2)}{q}[(q-2)(1+\varepsilon_{\ast})+\frac{5+\beta(b-1)}{b}]-\frac{\varepsilon}{b}}\ll1,
\end{align}	 
we further assume that
\begin{align}\label{I02}
2<q<2+\frac{-\tau_{3}+\sqrt{\tau_{3}^{2}+16b\varepsilon(1+\varepsilon_{\ast})}}{4b(1+\varepsilon_{\ast})},
\end{align}
where $\tau_{3}:=10-\varepsilon+2\beta(b-1)$. Then a consequence of \eqref{M012}--\eqref{I02}, Lemmas \ref{lem06} and \ref{lem26} shows that if $a\gg1$,
\begin{align}\label{W99}
\|w_{m+1}^{(p)}\|_{L^{q}}\leq& \sum\limits_{i\geq0}\sum\limits_{\zeta\in\Lambda_{(i)}}\big(\|a_{(\zeta)}\|_{L^{q}}\|W_{(\zeta)}\|_{L^{q}}+C_{q}(\lambda_{m+1}\sigma)^{-1/q}\|a_{(\zeta)}\|_{C^{1}}\|W_{(\zeta)}\|_{L^{q}}\big)\notag\\
\leq&\frac{\kappa_{\ast}}{6}\delta_{m+1}^{1/2}+C\delta_{m+1}^{1/2}\lambda_{m}^{-\varepsilon}\sum^{i_{\max}}_{i=1}2^{\frac{(q-2)i}{q}}+C\delta_{m+1}^{1/2}\ell^{-10}(\lambda_{m+1}\sigma)^{-1/q}\notag\\
\leq&\kappa_{\ast}\delta_{m+1}^{1/2}\Big(\frac{1}{3}+2^{\frac{(q-2)i_{\max}}{q}}\lambda_{m}^{-\varepsilon}\Big)\leq \frac{2}{3}\kappa_{\ast}\delta_{m+1}^{1/2}.
\end{align}	
Using \eqref{M01}, Lemma \ref{lem06} and the fact of $i_{\max}\ell\ll1$,  we deduce that for $k\geq1$,
\begin{align}\label{D90}
\|w_{m+1}^{(p)}\|_{W^{k,q}}\lesssim	&\sum\limits_{i\geq0}\sum\limits_{\zeta\in\Lambda_{(i)}}\sum_{j=0}^{k}\|a_{(\zeta)}\|_{C^{k-j}}\|W_{(\zeta)}\|_{W^{j,q}}\notag\\
\lesssim&\sum\limits_{i=0}^{i_{\max}}\sum\limits_{\zeta\in\Lambda_{(i)}}\sum_{j=0}^{k}\delta_{m+1}^{1/2}\ell^{-k+j-8}\lambda_{m+1}^{j}\lesssim \ell^{-9}\delta_{m+1}^{1/2}\lambda_{m+1}^{k}.
\end{align}
Similarly, for $k\geq0$,
\begin{align*}
&\big\|\mathrm{curl}(\nabla a_{(\zeta)}\times V_{(\zeta)})+\frac{1}{(N_{\Lambda}\lambda_{m+1})^{2}}\nabla(a_{(\zeta)}V_{(\zeta)})\times \mathrm{curl}(\Phi_{(\zeta)}\zeta)\big\|_{W^{k,q}}\notag\\
&\lesssim\sum_{j=0}^{k+1}\|a_{(\zeta)}\|_{C^{k+2-j}}\|V_{(\zeta)}\|_{W^{j,q}}\notag\\
&\quad+\frac{1}{\lambda_{m+1}^{2}}\sum_{j=0}^{k}\sum^{k-j+1}_{j'=0}\|a_{(\zeta)}\|_{C^{k-j+1-j'}}\|\psi_{(\zeta)}\|_{W^{j',q}}\|\Phi_{(\zeta)}\|_{W^{j+1,q}}\notag\\
&\lesssim\sum_{j=0}^{k+1}\delta_{m+1}^{1/2}\ell^{-k-2+j-8}\lambda_{m+1}^{j-2}\notag\\
&\quad+\sum_{j=0}^{k}\sum^{k-j+1}_{j'=0}\delta_{m+1}^{1/2}\ell^{-k+j+j'-9}\Big(\frac{\lambda_{m+1}\sigma}{r}\Big)^{j'}\lambda_{m+1}^{j-1}\lesssim\ell^{-8}\delta_{m+1}^{1/2}\lambda_{m+1}^{k-1},
\end{align*}	
and
\begin{align*}
&\big\|\mu^{-1}\mathbb{P}_{H}\mathbb{P}_{\neq0}\big(a_{(\zeta)}^{2}\phi_{(\zeta)}^{2}\psi_{(\zeta)}^{2}\zeta\big)\big\|_{W^{k,q}}\notag\\
&\lesssim\mu^{-1}\sum^{k}_{j=0}\sum^{j}_{j'=0}\big\|a^{2}_{(\zeta)}\big\|_{C^{k-j}}\big\|\phi^{2}_{(\zeta)}\big\|_{W^{j',q}}\big\|\psi^{2}_{(\zeta)}\big\|_{W^{j-j',q}}\notag\\
&\lesssim\mu^{-1}\sum^{k}_{j=0}\sum^{j}_{j'=0}\delta_{m+1}\ell^{-k+j-16}\lambda_{m+1}^{j'}\vartheta_{m+1}^{-1/q}\Big(\frac{\lambda_{m+1}\sigma}{r}\Big)^{j-j'}\notag\\
&\lesssim\ell^{-16}\mu^{-1}\delta_{m+1}\vartheta_{m+1}^{-1/q}\lambda_{m+1}^{k}.
\end{align*}	
Observe that if 
\begin{align}\label{I03}
2<q\leq2+\frac{\varepsilon}{34(1+\varepsilon_{\ast})}\;\text{ and }\;b>\frac{34(5+\beta(b-1))}{\varepsilon},
\end{align}
then $\ell^{-17}\lambda_{m+1}^{-\varepsilon_{\ast}}\ll\ell^{-17}\lambda_{m+1}^{-\varepsilon}\ll1$. Therefore, we obtain
\begin{align*}
\|w_{m+1}^{(c)}\|_{L^{q}}+\|w_{m+1}^{(t)}\|_{L^{q}}\lesssim&\ell^{-9}\delta_{m+1}^{1/2}\lambda_{m+1}^{-1}+\ell^{-17}\mu^{-1}\delta_{m+1}\vartheta_{m+1}^{-1/q}\notag\\
\leq&C\delta_{m+1}^{1/2}(\ell^{-9}\lambda_{m+1}^{-1}+\ell^{-17}\lambda_{m+1}^{-\varepsilon_{\ast}-\beta})\leq\frac{\kappa_{\ast}}{6}\delta_{m+1}^{1/2},
\end{align*}	
and, for $k\geq1$,
\begin{align*}
\|w_{m+1}^{(c)}\|_{W^{k,q}}+\|w_{m+1}^{(t)}\|_{W^{k,q}}\lesssim&\ell^{-9}\delta_{m+1}^{1/2}\lambda_{m+1}^{k-1}+\ell^{-17}\mu^{-1}\delta_{m+1}\vartheta_{m+1}^{-1/q}\lambda_{m+1}^{k}\notag\\
\lesssim&\ell^{-17}\delta_{m+1}^{1/2}\lambda_{m+1}^{k-\varepsilon_{\ast}-\beta}\leq\delta_{m+1}^{1/2}\lambda_{m+1}^{k},
\end{align*}
where we used the fact of $\mu^{-1}\vartheta_{m+1}^{-1/q}=\lambda_{m+1}^{-\varepsilon_{\ast}}$ in virtue of \eqref{QZ003} and \eqref{ZQ009}. These facts, together with \eqref{W99}--\eqref{I03} and Lemma \ref{lem02}, read that
\begin{align*}
	\|u_{m+1}-u_{m}\|_{L^{q}}\leq& \|u_{\ell}-u_{m}\|_{L^{q}}+\|w_{m+1}\|_{L^{q}}\notag\\
	\leq&C\delta_{m+1}^{1/2}\lambda_{m}^{-4+\varepsilon}\vartheta_{m+1}^{\frac{q-2}{q}}+\frac{5\kappa_{\ast}}{6}\delta_{m+1}^{1/2}\leq \kappa_{\ast}\delta_{m+1}^{1/2},
\end{align*}	
and
\begin{align*}
\|w_{m+1}\|_{W^{k,q}}\leq&C\ell^{-9}\delta_{m+1}^{1/2}\lambda_{m+1}^{k}\leq\frac{1}{2}\delta_{m+1}^{1/2}\lambda_{m+1}^{k+\varepsilon}.
\end{align*}

By the same arguments as before, it follows from a straightforward calculation that
\begin{align*}
\|w_{m+1}^{(p)}\|_{C_{x,t}^{1}}\lesssim	&\sum\limits_{i\geq0}\sum\limits_{\zeta\in\Lambda_{(i)}}\|a_{(\zeta)}\|_{C^{1}_{x,t}}\|W_{(\zeta)}\|_{C^{1}_{x,t}}\notag\\
\lesssim&\sum\limits_{i=0}^{i_{\max}}\sum\limits_{\zeta\in\Lambda_{(i)}}\delta_{m+1}^{1/2}\ell^{-9}\vartheta_{m+1}^{-1/q}\frac{\lambda_{m+1}\sigma\mu}{r}\lesssim \ell^{-10}\delta_{m+1}^{1/2}\lambda_{m+1}^{3+\varepsilon_{\ast}}\leq\frac{1}{6}\lambda_{m+1}^{4},
\end{align*}
and
\begin{align*}
\|w_{m+1}^{(c)}\|_{C_{x,t}^{1}}\lesssim&i_{\max}\|a_{(\zeta)}\|_{C_{x,t}^{2}}\|\nabla V_{(\zeta)}\|_{C_{x,t}^{1}}\notag\\
&+\frac{i_{\max}}{\lambda_{m+1}^{2}}\|a_{(\zeta)}\|_{C_{x,t}^{2}}\big(\|\nabla V_{(\zeta)}\|_{C_{x,t}^{1}}\|\nabla\Phi\|_{C^{0}_{x,t}}+\|\nabla V_{(\zeta)}\|_{C_{x,t}^{0}}\|\nabla\Phi\|_{C^{1}_{x,t}}\big)\notag\\
\lesssim&\ell^{-11}\delta_{m+1}^{1/2}\vartheta_{m+1}^{-1/q}\frac{\sigma\mu}{r}\lesssim\ell^{-11}\delta_{m+1}^{1/2}\lambda_{m+1}^{2+\varepsilon_{\ast}}\leq\frac{1}{6}\lambda_{m+1}^{4},
\end{align*}
and
\begin{align*}
\|w_{m+1}^{(t)}\|_{C_{x,t}^{1}}\lesssim&i_{\max}\mu^{-1}\|a_{(\zeta)}\|_{C_{x,t}^{1}}\|a_{(\zeta)}\|_{C_{x,t}^{0}}\|\phi_{(\zeta)}\psi_{(\zeta)}\|_{C_{x,t}^{1}}\|\phi_{(\zeta)}\psi_{(\zeta)}\|_{C_{x,t}^{0}}\notag\\
\lesssim&\ell^{-18}\delta_{m+1}\vartheta_{m+1}^{-2/q}\frac{\lambda_{m+1}\sigma}{r}\lesssim\ell^{-18}\delta_{m+1}\lambda_{m+1}^{3-\varepsilon_{\ast}}\leq\frac{1}{6}\lambda_{m+1}^{4}.
\end{align*}		
Consequently, $\|w_{m+1}\|_{C_{x,t}^{1}}\leq\frac{1}{2}\lambda_{m+1}^{4}$.

\end{proof}

A direct application of Lemma \ref{lem02} and Proposition \ref{prop06} gives that the results in \eqref{Z09} hold with $m$ replaced by $m+1$. To be specific, 
\begin{corollary}\label{coro03}
There exists a constant $2<q\ll3$ such that
\begin{align*}
	\|u_{m+1}\|_{C_{t}L^{q}}\leq2\delta_{0}^{1/2}-\delta_{m+1}^{1/2},\quad\|\nabla u_{m+1}\|_{C_{t}L^{q}}\leq\lambda_{m+1}^{1+\varepsilon}\delta_{m+1}^{1/2},\quad\|u_{m+1}\|_{C^{1}_{x,t}}\leq\lambda_{m+1}^{4}.
\end{align*}		
	
\end{corollary}	
\begin{proof}
If $q\in(2,3)$ satisfies \eqref{I01}, \eqref{I02} and \eqref{I03}, we deduce from \eqref{Z09}, Lemma \ref{lem02} and Proposition \ref{prop06} that
\begin{align*}
\|u_{m+1}\|_{L^{q}}\leq&\|u_{m}\|_{L^{q}}+\|u_{\ell}-u_{m}\|_{L^{q}}+\|w_{m+1}\|_{L^{q}}\notag\\
\leq&2\delta_{0}^{1/2}-\delta_{m}^{1/2}+C\delta_{m+1}^{1/2}\lambda_{m}^{-4+\varepsilon}\vartheta_{m+1}^{\frac{q-2}{q}}+\kappa_{\ast}\delta_{m+1}^{1/2}\leq2\delta_{0}^{1/2}-\delta_{m+1}^{1/2},\notag\\
\|\nabla u_{m+1}\|_{L^{q}}\leq&\|\nabla w_{m+1}\|_{L^{q}}+\|\nabla u_{\ell}\|_{L^{q}}\leq\frac{1}{2}\delta_{m+1}^{1/2}\lambda_{m+1}^{1+\varepsilon}+C\lambda_{m}^{1+\varepsilon}\delta^{1/2}_{m}\leq\lambda_{m+1}^{1+\varepsilon}\delta_{m+1}^{1/2},\notag\\
\|u_{m+1}\|_{C^{1}_{x,t}}\leq&\|u_{m+1}\|_{C^{1}_{x,t}}+\|u_{\ell}\|_{C^{1}_{x,t}}\leq\|u_{m+1}\|_{C^{1}_{x,t}}+\|u_{m}\|_{C^{1}_{x,t}}\notag\\
\leq&\frac{1}{2}\lambda_{m+1}^{4}+\lambda_{m}^{4}\leq \lambda_{m+1}^{4},
\end{align*}	
where we also utilized Young's convolution inequality.

\end{proof}

\subsection{Estimates on the Reynolds stress}
For the purpose of specifying the new Reynolds stress, we first recall the antidivergence operator $\mathcal{R}$ introduced in \cite{DS2013}.
\begin{definition}
Set $n\geq3$. For any smooth $u\in C^{\infty}(\mathbb{T}^{n},\mathbb{R}^{n})$, define the operator $\mathcal{R}:C^{\infty}(\mathbb{T}^{n},\mathbb{R}^{n})\rightarrow C^{\infty}(\mathbb{T}^{n},\mathcal{S}_{0}^{n\times n})$ by
\begin{align*}
\mathcal{R}u=\frac{n-2}{2n-2}\big(\nabla\mathbb{P}_{H}v+(\nabla\mathbb{P}_{H}v)^{T}\big)+\frac{n}{2n-2}\big(\nabla v+(\nabla v)^{T}\big)-\frac{1}{n-1}\mathrm{div}v\mathrm{Id},
\end{align*}		
where $\mathcal{S}_{0}^{n\times n}$ represents the space whose elements consist of trace-free symmetric $n\times n$ matrices, $\mathbb{P}_{H}$ is the Leray projection, $v\in C^{\infty}(\mathbb{T}^{n},\mathbb{R}^{n})$ is the solution of 
$$\Delta v=u-\dashint_{\mathbb{T}^{n}}udx,\quad\mathrm{with}\;\dashint_{\mathbb{T}^{n}}vdx=0.$$
\end{definition}	
In light of the definition of $\mathcal{R}$, we have from a straightforward computation that $\mathrm{div}\mathcal{R}u=u-\dashint_{\mathbb{T}^{n}}u$ for $n\geq3$. Recall that $u_{m+1}=w_{m+1}+u_{\ell}$. Define the new Reynolds stress $\mathring{R}_{m+1}$ and scalar pressure $p_{m+1}$ as
\begin{align*}
\mathring{R}_{m+1}:=&\mathring{R}_{\mathrm{linear}}+\mathring{R}_{\mathrm{corrector}}+\mathring{R}_{\mathrm{oscillation}},\quad p_{m+1}:=p_{\ell}-P,
\end{align*}	
where $\mathring{R}_{\mathrm{linear}},\mathring{R}_{\mathrm{corrector}}$ are given by
\begin{align*}
\mathring{R}_{\mathrm{linear}}:=&\mathcal{R}\big(-\nu\Delta w_{m+1}+\partial_{t}(w_{m+1}^{(p)}+w_{m+1}^{(c)})+\mathrm{div}(u_{\ell}\otimes w_{m+1}+w_{m+1}\otimes u_{\ell})\big),\\
\mathring{R}_{\mathrm{corrector}}:=&\mathcal{R}\mathrm{div}\big((w_{m+1}^{(c)}+w_{m+1}^{(t)})\otimes w_{m+1}+w_{p+1}^{(p)}\otimes(w_{m+1}^{(c)}+w_{m+1}^{(t)})\big),
\end{align*}	
and $\mathring{R}_{\mathrm{oscillation}}, P$ are determined by
\begin{align}\label{AD90}
\mathrm{div} \mathring{R}_{\mathrm{oscillation}}+\nabla P=\partial_{t}w_{m+1}^{(t)}+\mathrm{div}(w_{m+1}^{(p)}\otimes w_{m+1}^{(p)}+\mathring{R}_{\ell}).
\end{align}	
Then $(u_{m+1},\mathring{R}_{m+1},p_{m+1})$ is a solution to the Navier-Stokes-Reynolds system. With regard to the estimates of $\mathring{R}_{m+1}$, we have
\begin{lemma}\label{lem10}
There exists some constant $2<q\ll3$ such that
\begin{align}\label{E85}
\|\mathring{R}_{m+1}\|_{L^{1}}\leq \vartheta_{m+2}^{\frac{q-2}{q}}\delta_{m+2}\lambda_{m+1}^{-2\varepsilon}.
\end{align}
\end{lemma}	

Before proving Lemma \ref{lem10}, we first recall a calculating lemma established in \cite{BV2019}.
\begin{lemma}[Lemma B.1 in \cite{BV2019}]\label{lem12}
Assume that $\kappa\in[1,\infty),p\in(1,2]$ and $M\in\mathbb{N}$ is sufficiently large. Pick $g\in C^{M}(\mathbb{T}^{3})$ satisfying that for all $0\leq j\leq M$, there holds $\|\nabla^{j}g\|_{L^{\infty}}\leq C_{g}\lambda^{j}$ for some $\lambda\in[1,\kappa]$ and $C_{g}>0$. Then for any $f\in L^{p}(\mathbb{T}^{3})$, if $\int_{\mathbb{T}^{3}}g(x)\mathbb{P}_{\geq\kappa}f(x)dx=0$, we have
\begin{align*}
\big\||\nabla|^{-1}(g\mathbb{P}_{\geq\kappa}f)\big\|_{L^{p}}\lesssim C_{g}\Big(1+\frac{\lambda^{M}}{\kappa^{M-2}}\Big)\frac{\|f\|_{L^{p}}}{\kappa},	
\end{align*}	  
where the implicit constant is only dependent of $p,M$.
\end{lemma}

\begin{proof}[Proof of Lemma \ref{lem10}]
{\bf{Step 1.}} {\textit{Estimates of $\mathring{R}_{\mathrm{linear}}$ and $\mathring{R}_{\mathrm{corrector}}$}}.
Similarly as in Proposition \ref{prop06}, it follows from the Calder\'{o}n-Zygmund estimates, \eqref{ZQ009}, \eqref{M01}, Lemmas \ref{lem02} and \ref{lem06}, the fact of $i_{\max}\ell\ll1$ that for $p\in(1,2)$ and $q\in(2,3)$, 
\begin{align}\label{DE12}
\|\mathring{R}_{\mathrm{linear}}\|_{L^{p}}\leq&\|\mathcal{R}(-\nu\Delta w_{m+1})\|_{L^{p}}+\|\mathcal{R}(\partial_{t}(w_{m+1}^{(p)}+w_{m+1}^{(c)}))\|_{L^{p}}\notag\\
&+\|\mathcal{R}\mathrm{div}(u_{\ell}\otimes w_{m+1}+w_{m+1}\otimes u_{\ell})\|_{L^{p}}\notag\\
\lesssim&\|w_{m+1}\|_{W^{1,p}}+\sum_{i}\sum_{\zeta\in\Lambda_{(i)}}\|\partial_{t}\mathcal{R}\mathrm{curl}\mathrm{curl}(a_{(\zeta)}V_{(\zeta)})\|_{L^{p}}+\|u_{\ell}\|_{L^{\infty}}\|w_{m+1}\|_{L^{p}}\notag\\
\lesssim&\|w_{m+1}\|_{W^{1,p}}+\|u_{\ell}\|_{W^{3,p}}\|w_{m+1}\|_{L^{p}}\notag\\
&+\sum_{i}\sum_{\zeta\in\Lambda_{(i)}}\big(\|a_{(\zeta)}\|_{C^{1}_{x,t}}\|\partial_{t}V_{(\zeta)}\|_{W^{1,p}}+\|a_{(\zeta)}\|_{C^{2}_{x,t}}\|V_{(\zeta)}\|_{W^{1,p}}\big)\notag\\
\lesssim&\ell^{-17}\delta_{m+1}^{1/2}\lambda_{m+1}\vartheta_{m+1}^{1/p-1/q}+\ell^{-2}\delta^{1/2}_{m}\delta_{m+1}^{1/2}\lambda_{m}^{1+\varepsilon}\vartheta_{m+1}^{1/p-1/q}\notag\\
&+\ell^{-10}\delta_{m+1}^{1/2}\frac{\sigma\mu}{r}\vartheta_{m+1}^{1/p-1/q}+\ell^{-11}\delta_{m+1}^{1/2}\lambda_{m+1}^{-1}\vartheta_{m+1}^{1/p-1/q}\notag\\
\lesssim& \ell^{-17}\delta_{m+1}^{1/2}\vartheta_{m+1}^{1/p-2/q}\lambda_{m+1}^{-\varepsilon_{\ast}},
\end{align}	
and 
\begin{align}\label{DE13}
\|\mathring{R}_{\mathrm{corrector}}\|_{L^{p}}\lesssim&\|(w_{m+1}^{(c)}+w_{m+1}^{(t)})\otimes w_{m+1}+w_{p+1}^{(p)}\otimes(w_{m+1}^{(c)}+w_{m+1}^{(t)})\|_{L^{p}}	\notag\\
\lesssim&\|w_{m+1}^{(c)}+w_{m+1}^{(t)}\|_{L^{2p}}\|w_{m+1}\|_{L^{2p}}+\|w_{m+1}^{(p)}\|_{L^{2p}}\|w_{m+1}^{(c)}+w_{m+1}^{(t)}\|_{L^{2p}}\notag\\
\lesssim&\ell^{-35}\delta_{m+1}\mu^{-1}\vartheta_{m+1}^{1/p-3/q}\lesssim \ell^{-35}\delta_{m+1}\vartheta_{m+1}^{1/p-2/q}\lambda_{m+1}^{-\varepsilon_{\ast}}.
\end{align}	
{\bf{Step 2.}} {\textit{Estimate of $\mathring{R}_{\mathrm{oscillation}}$.}}
According to \eqref{AD90}, we first find exact expressions of $\mathring{R}_{\mathrm{oscillation}}$ and $P$. In light of \eqref{ZQ10} and \eqref{ZQ29}, we have
\begin{align*}
&\mathrm{div}(w_{m+1}^{(p)}\otimes w_{m+1}^{(p)}+\mathring{R}_{\ell})=\sum_{i,j}\sum_{\zeta\in\Lambda_{(i)},\zeta'\in\Lambda_{(j)}}\mathrm{div}\big(a_{(\zeta)}a_{(\zeta')}W_{(\zeta)}\otimes W_{(\zeta')}\big)+\mathrm{div}\mathring{R}_{\ell}\notag\\
&=\sum_{i}\sum_{\zeta\in\Lambda_{(i)}}\mathrm{div}\big(a_{(\zeta)}^{2}W_{(\zeta)}\otimes W_{(\zeta)}\big)+\mathrm{div}\mathring{R}_{\ell}\notag\\
&=\sum_{i}\sum_{\zeta\in\Lambda_{(i)}}\mathrm{div}\Big(a_{(\zeta)}^{2}\Big(W_{(\zeta)}\otimes W_{(\zeta)}-\dashint_{\mathbb{T}^{3}}W_{(\zeta)}\otimes W_{(\zeta)}dx\Big)\Big)+\nabla\Big(\sum_{i\geq0}\rho_{i}\chi_{(i)}^{2}\Big)\notag\\
&=\sum_{i}\sum_{\zeta\in\Lambda_{(i)}}\mathrm{div}\big(a_{(\zeta)}^{2}\mathbb{P}_{\geq N_{\Lambda}\lambda_{m+1}\sigma}(W_{(\zeta)}\otimes W_{(\zeta)})\big)+\nabla\Big(\sum_{i\geq0}\rho_{i}\chi_{(i)}^{2}\Big).
\end{align*}	
For simplicity, denote $\mathcal{A}_{(\zeta)}:=\mathrm{div}\big(a_{(\zeta)}^{2}\mathbb{P}_{\geq N_{\Lambda}\lambda_{m+1}\sigma}(W_{(\zeta)}\otimes W_{(\zeta)})\big)$. We further decompose it into two parts $\mathcal{A}_{(\zeta,1)}$ and $\mathcal{A}_{(\zeta,2)}$ given by
\begin{align*}
\begin{cases}
\mathcal{A}_{(\zeta,1)}=\mathbb{P}_{\neq0}\big(\nabla(a^{2}_{(\zeta)})\mathbb{P}_{\geq N_{\Lambda}\lambda_{m+1}\sigma}(W_{(\zeta)}\otimes W_{(\zeta)})\big),\\
\mathcal{A}_{(\zeta,2)}=\mathbb{P}_{\neq0}\big(a^{2}_{(\zeta)}\mathrm{div}(W_{(\zeta)}\otimes W_{(\zeta)})\big).
\end{cases}
\end{align*}	 
Observe from the definition of $\psi_{(\zeta)}$ that $(\zeta\cdot\nabla)\psi_{(\zeta)}=\mu^{-1}\partial_{t}\psi_{(\zeta)}.$ This leads to the following identity
\begin{align*}
\mathrm{div}(W_{(\zeta)}\otimes W_{(\zeta)})=2\big(W_{(\zeta)}\cdot\nabla\psi_{(\zeta)}\big)\phi_{(\zeta)}\zeta=\mu^{-1}\phi_{(\zeta)}^{2}\partial_{t}\psi_{(\zeta)}^{2}\zeta.
\end{align*}
Substituting this identity into $\mathcal{A}_{(\zeta,2)}$, we deduce
\begin{align}\label{E09}
\mathcal{A}_{(\zeta,2)}=\mu^{-1}\partial_{t}\mathbb{P}_{\neq0}\big(a^{2}_{(\zeta)}\phi^{2}_{(\zeta)}\psi^{2}_{(\zeta)}\zeta\big)-\mu^{-1}\mathbb{P}_{\neq0}\big((\partial_{t}a^{2}_{(\zeta)})\phi^{2}_{(\zeta)}\psi^{2}_{(\zeta)}\zeta\big).
\end{align}	
For convenience, write $\mathcal{A}_{(\zeta,3)}:=-\mu^{-1}\mathbb{P}_{\neq0}\big((\partial_{t}a^{2}_{(\zeta)})\phi^{2}_{(\zeta)}\psi^{2}_{(\zeta)}\zeta\big).$ Recalling the definition of $w_{m+1}^{(t)}$ and in view of $\mathbb{P}_{H}=\mathrm{Id}-\nabla\Delta^{-1}\mathrm{div}$, we have from \eqref{E09} that
\begin{align*}
&\sum_{i}\sum_{\zeta\in\Lambda_{(i)}}\mathcal{A}_{(\zeta,2)}+\partial_{t}w_{m+1}^{(t)}\notag\\
&=\mu^{-1}\sum_{i}\sum_{\zeta\in\Lambda_{(i)}}(\mathrm{Id}-\mathbb{P}_{H})\partial_{t}\mathbb{P}_{\neq0}\big(a^{2}_{(\zeta)}\phi^{2}_{(\zeta)}\psi^{2}_{(\zeta)}\zeta\big)+\sum_{i}\sum_{\zeta\in\Lambda_{(i)}}\mathcal{A}_{(\zeta,3)}\notag\\
&=\nabla\Big(\mu^{-1}\sum_{i}\sum_{\zeta\in\Lambda_{(i)}}\Delta^{-1}\mathrm{div}\partial_{t}\mathbb{P}_{\neq0}\big(a^{2}_{(\zeta)}\phi^{2}_{(\zeta)}\psi^{2}_{(\zeta)}\zeta\big)\Big)+\sum_{i}\sum_{\zeta\in\Lambda_{(i)}}\mathcal{A}_{(\zeta,3)}.
\end{align*}	
Combining these above facts, we see that the explicit forms of $\mathring{R}_{\mathrm{oscillation}}$ and $P$ are, respectively, given by
\begin{align*}
\mathring{R}_{\mathrm{oscillation}}=&\sum_{i}\sum_{\zeta\in\Lambda_{(i)}}\mathcal{R}(\mathcal{A}_{(\zeta,1)}+\mathcal{A}_{(\zeta,3)}),
\end{align*}
and
\begin{align*}
 P=&\sum_{i\geq0}\rho_{i}\chi_{(i)}^{2}+\mu^{-1}\sum_{i}\sum_{\zeta\in\Lambda_{(i)}}\Delta^{-1}\mathrm{div}\partial_{t}\mathbb{P}_{\neq0}\big(a^{2}_{(\zeta)}\phi^{2}_{(\zeta)}\psi^{2}_{(\zeta)}\zeta\big).
\end{align*}	
Applying Lemma \ref{lem12} with $\lambda=\ell^{-1},$ $C_{g}=\ell^{-17}\delta_{m+1}$ and $\kappa=N_{\Lambda}\lambda_{m+1}\sigma$, we obtain that for a sufficiently large $L>0$,
\begin{align*}
\|\mathcal{R}\mathcal{A}_{(\zeta,1)}\|_{L^{p}}\lesssim&\big\||\nabla|^{-1}\mathbb{P}_{\neq0}\big(\nabla(a^{2}_{(\zeta)})\mathbb{P}_{\geq N_{\Lambda}\lambda_{m+1}\sigma}(W_{(\zeta)}\otimes W_{(\zeta)})\big)\big\|_{L^{p}}\notag\\
\lesssim&\frac{\ell^{-17}\delta_{m+1}}{\lambda_{m+1}\sigma}\bigg(1+\frac{1}{\ell^{L}(\lambda_{m+1}\sigma)^{L-2}}\bigg)\|W_{(\zeta)}\otimes W_{(\zeta)}\|_{L^{p}}\notag\\
\lesssim&\frac{\ell^{-17}\delta_{m+1}}{\lambda_{m+1}\sigma}\|\phi_{(\zeta)}\|_{L^{2p}}^{2}\|\psi_{(\zeta)}\|_{L^{2p}}^{2}\lesssim \frac{\ell^{-17}\delta_{m+1}\vartheta_{m+1}^{1/p-2/q}}{\lambda_{m+1}\sigma}.
\end{align*}	
Analogously, we have
\begin{align*}
\|\mathcal{R}\mathcal{A}_{(\zeta,3)}\|_{L^{p}}\lesssim&\mu^{-1}\|(\partial_{t}a^{2}_{(\zeta)})\phi^{2}_{(\zeta)}\psi^{2}_{(\zeta)}\zeta\|_{L^{p}}\notag\\
\lesssim&\mu^{-1}\|a_{(\zeta)}\|_{C^{1}_{t,x}}\|a_{(\zeta)}\|_{L^{\infty}}\|\phi_{(\zeta)}\|_{L^{2p}}^{2}\|\psi_{(\zeta)}\|_{L^{2p}}^{2}\notag\\
\lesssim&\mu^{-1}\ell^{-17}\delta_{m+1}\vartheta_{m+1}^{1/p-2/q}.
\end{align*}	
Therefore, we have
\begin{align*}
\|\mathring{R}_{\mathrm{oscillation}}\|_{L^{p}}\lesssim&\sum_{i}\sum_{\zeta\in\Lambda_{(i)}}\big(\|\mathcal{R}\mathcal{A}_{(\zeta,1)}\|_{L^{p}}+\|\mathcal{R}\mathcal{A}_{(\zeta,3)}\|_{L^{p}}\big)\notag\\
\lesssim&\frac{\ell^{-17}\delta_{m+1}\vartheta_{m+1}^{1/p-2/q}}{\lambda_{m+1}\sigma},
\end{align*}	
which, in combination with \eqref{DE12}--\eqref{DE13} and H\"{o}lder's inequality, reads that
\begin{align*}
\|\mathring{R}_{m+1}\|_{L^{1}}\lesssim\|\mathring{R}_{m+1}\|_{L^{p}}\leq \ell^{-36}\delta_{m+1}^{1/2}\vartheta_{m+1}^{1/p-2/q}\lambda_{m+1}^{-\varepsilon_{\ast}}.
\end{align*}	
To achieve \eqref{E85}, it suffices to ensure that
\begin{align*}
\ell^{-36}\delta_{m+1}^{1/2}\vartheta_{m+1}^{1/p-2/q}\lambda_{m+1}^{-\varepsilon_{\ast}}\leq\vartheta_{m+2}^{\frac{q-2}{q}}\delta_{m+2}\lambda_{m+1}^{-2\varepsilon},
\end{align*}	
which is equivalent to requiring
\begin{align*}
\theta:=&36\Big((q-2)(1+\varepsilon_{\ast})+\frac{5+\beta(b-1)}{b}\Big)-\beta+\frac{q(1+\varepsilon)(p-1)}{p}-\varepsilon_{\ast}\notag\\
\leq&-2b\beta-2\varepsilon-(b-1)(q-2)(1+\varepsilon_{\ast}).	
\end{align*}	
Note that
\begin{align*}
&\theta+2b\beta+\varepsilon_{\ast}+(b-1)(q-2)(1+\varepsilon_{\ast})\notag\\
&\leq(b+35)(q-2)(1+\varepsilon_{\ast})+\frac{180}{b}+(36+2b)\beta+q(1+\varepsilon)(p-1).
\end{align*}	
Therefore, if the parameters $p,q,b,\beta$ satisfy
\begin{align}\label{D50}
\begin{cases}
2<q\leq2+\frac{\varepsilon_{\ast}-2\varepsilon}{4(b+35)(1+\varepsilon_{\ast})},\quad b\geq\frac{720}{\varepsilon_{\ast}-2\varepsilon},\\
0<\beta\leq\frac{\varepsilon_{\ast}-2\varepsilon}{4(36+2b)},\quad1<p\leq1+\frac{\varepsilon_{\ast}-2\varepsilon}{q(1+\varepsilon)},
\end{cases}
\end{align}	
then $\theta+2b\beta+\varepsilon_{\ast}+(b-1)(q-2)(1+\varepsilon_{\ast})\leq\varepsilon_{\ast}-2\varepsilon$ and thus \eqref{E85} is proved.

\end{proof}

\subsection{The energy iteration}
This section is devoted to proving that \eqref{QZ005} holds with $m$ replaced by $m+1.$ 
\begin{lemma}\label{lem20}
There exists some constant $2<q\ll3$ such that	
\begin{align}\label{EI001}
		0\leq e(t)-\int_{\mathbb{T}^{3}}|u_{m+1}|^{2}\leq\vartheta_{m+2}^{\frac{q-2}{q}}\delta_{m+2},	
\end{align}
and
\begin{align}\label{EI002}
e(t)-\int_{\mathbb{T}^{3}}|u_{m+1}|^{2}\leq\frac{\vartheta_{m+2}^{\frac{q-2}{q}}\delta_{m+2}}{10}\;\Longrightarrow\;u_{m+1}(\cdot,t)\equiv0,\;\mathring{R}_{m+1}(\cdot,t)\equiv0.
\end{align}	
\end{lemma}
\begin{proof}
{\bf{Step 1.}} Claim that 
\begin{align}\label{W88}
	e(t)-\int_{\mathbb{T}^{3}}|u_{m+1}|^{2}dx=&\tilde{e}(t)-3\rho_{0}\int_{\mathbb{T}^{3}}\chi_{(0)}^{2}dx+o(1)\vartheta_{m+2}^{\frac{q-2}{q}}\delta_{m+2},
\end{align}	
where $|o(1)|\leq\frac{1}{100}$, and
\begin{align*}
	\tilde{e}(t)=e(t)-\int_{\mathbb{T}^{3}}|u_{m}|^{2}dx-3\sum_{i\geq1}\rho_{i}\int_{\mathbb{T}^{3}}\chi^{2}_{(i)}dx.
\end{align*}	
Observe from Lemma \ref{lem05} that
\begin{align}\label{E906}
	&\sum_{i\geq0}\sum_{\zeta\in\Lambda_{(i)}}\mathrm{tr}\bigg(\int_{\mathbb{T}^{3}}a_{(\zeta)}^{2}dx\dashint_{\mathbb{T}^{3}}W_{(\zeta)}\otimes W_{(\zeta)}dx\bigg)\notag\\
	&=\sum_{i\geq0}\rho_{i}\chi_{(i)}^{2}\mathrm{tr}\int_{\mathbb{T}^{3}}\sum_{\zeta\in\Lambda_{(i)}}\gamma_{(\zeta)}^{2}\Big(\mathrm{Id}-\frac{\mathring{R}_{\ell}}{\rho_{i}}\Big)\zeta\otimes\zeta dx\notag\\
	&=\sum_{i\geq0}\rho_{i}\chi_{(i)}^{2}\int_{\mathbb{T}^{3}}\mathrm{tr}\Big(\mathrm{Id}-\frac{\mathring{R}_{\ell}}{\rho_{i}}\Big)dx=3\sum_{i\geq0}\rho_{i}\int_{\mathbb{T}^{3}}\chi_{(i)}^{2}dx,
\end{align}
which gives that		
\begin{align*}
\int_{\mathbb{T}^{3}}\big|w_{m+1}^{(p)}\big|^{2}dx=&\mathrm{tr}\int_{\mathbb{T}^{3}}w_{m+1}^{(p)}\otimes w_{m+1}^{(p)}dx=\sum_{i\geq0}\sum_{\zeta\in\Lambda_{(i)}}\mathrm{tr}\int_{\mathbb{T}^{3}}a_{(\zeta)}^{2}W_{(\zeta)}\otimes W_{(\zeta)}dx\notag\\
=&3\sum_{i\geq0}\rho_{i}\int_{\mathbb{T}^{3}}\chi_{(i)}^{2}dx+\sum_{i\geq0}\sum_{\zeta\in\Lambda_{(i)}}\int_{\mathbb{T}^{3}}a_{(\zeta)}^{2}\mathbb{P}_{\geq N_{\Lambda}\lambda_{m+1}\sigma}|W_{(\zeta)}|^{2}dx,
\end{align*}
where $\mathrm{tr}$ denotes the trace operation satisfying that $\mathrm{tr}B=\sum_{i=1}^{3}b_{ii}$ for any matrix $B=(b_{ij})_{3\times3}$.  Making use of integration by parts, we obtain that if $M\gg1$ and $q\in(2,3)$ satisfies \eqref{I01} and \eqref{I03},
\begin{align*}
&\bigg|\sum_{i\geq0}\sum_{\zeta\in\Lambda_{(i)}}\int_{\mathbb{T}^{3}}a_{(\zeta)}^{2}\mathbb{P}_{\geq N_{\Lambda}\lambda_{m+1}\sigma}|W_{(\zeta)}|^{2}dx\bigg|\notag\\
&\leq\sum_{i\geq0}\sum_{\zeta\in\Lambda_{(i)}}\bigg|\int_{\mathbb{T}^{3}}|\nabla|^{M}a_{(\zeta)}^{2}|\nabla|^{-M}\mathbb{P}_{N_{\Lambda}\lambda_{m+1}\sigma}|W_{(\zeta)}|^{2}dx\bigg|\notag\\
&\lesssim\sum_{i\geq0}\sum_{\zeta\in\Lambda_{(i)}}(\lambda_{m+1}\sigma)^{-M}\|a_{(\zeta)}^{2}\|_{C^{M}}\|W_{(\zeta)}\|_{L^{2}}\notag\\
&\lesssim\delta_{m+1}\ell^{-M-17}(\lambda_{m+1}\sigma)^{-M}\vartheta_{m+1}^{\frac{q-2}{q}}\leq\frac{\vartheta_{m+2}^{\frac{q-2}{q}}\delta_{m+2}}{\lambda_{m+1}^{2\beta}},
\end{align*}
and
\begin{align*}
&\bigg|\int_{\mathbb{T}^{3}}w_{m+1}\cdot u_{\ell}dx\bigg|\lesssim\big\|w_{m+1}^{(t)}\big\|_{L^{1}}\|u_{\ell}\|_{L^{\infty}}+\|a_{(\zeta)}V_{(\zeta)}\|_{L^{1}}\|u_{\ell}\|_{C^{2}}\notag\\
&\lesssim\big\|w_{m+1}^{(t)}\big\|_{L^{1}}\|u_{\ell}\|_{W^{2,q}}+\|a_{(\zeta)}\|_{L^{\infty}}\|V_{(\zeta)}\|_{L^{1}}\|u_{\ell}\|_{C^{2}}\notag\\
&\lesssim\ell^{-19}\lambda_{m}^{1+\varepsilon}\delta_{m}^{\frac{1}{2}}\delta_{m+1}\mu^{-1}\vartheta_{m+1}^{\frac{q-2}{q}}+\ell^{-10}\lambda_{m}^{1+\varepsilon}\delta_{m}^{\frac{1}{2}}\delta_{m+1}^{\frac{1}{2}}\frac{\vartheta_{m+1}^{1-1/q}}{\lambda_{m+1}^{2}}\leq\frac{\vartheta_{m+2}^{\frac{q-2}{q}}\delta_{m+2}}{\lambda_{m}^{\beta}}.
\end{align*}		
Applying \eqref{DE13} with $p=1$, we deduce that if $q\in(2,3)$ satisfies \eqref{D50},
\begin{align*}
&\bigg|\int_{\mathbb{T}^{3}}|w_{m+1}|^{2}dx-\int_{\mathbb{T}^{3}}\big|w_{m+1}^{(p)}\big|^{2}dx\bigg|\leq\big(\|w_{m+1}\|_{L^{2}}+\big\|w_{m+1}^{(p)}\big\|_{L^{2}}\big)\big\|w_{m+1}^{(c)}+w_{m+1}^{(t)}\big\|_{L^{2}}\notag\\
&\lesssim \ell^{-35}\delta_{m+1}\vartheta_{m+1}^{1-2/q}\lambda_{m+1}^{-\varepsilon_{\ast}}\leq\vartheta_{m+2}^{\frac{q-2}{q}}\delta_{m+2}\lambda_{m+1}^{-2\varepsilon}.
\end{align*}	
From Lemma \ref{lem02}, H\"{o}lder's inequality and Young's convolution inequality, it follows that if $0<b^{2}\beta<\frac{4-\varepsilon}{2}$ and $2<q<2+\frac{4-\varepsilon}{2b(b-1)(1+\varepsilon_{\ast})}$,
\begin{align}\label{F01}
&\bigg|\int_{\mathbb{T}^{3}}|u_{m}|^{2}dx-\int_{\mathbb{T}^{3}}|u_{\ell}|^{2}dx\bigg|\leq4\|u_{m}\|_{C_{t}L^{2}}\|u_{m}-u_{\ell}\|_{L^{2}}\notag\\
&\leq C\delta_{m+1}^{1/2}\lambda_{m}^{-4+\varepsilon}\vartheta_{m+1}^{\frac{q-2}{q}}\leq\frac{\vartheta_{m+2}^{\frac{q-2}{q}}\delta_{m+2}}{\lambda_{m+1}^{\beta}}.
\end{align}
A consequence of these above facts reads that \eqref{W88} holds.

{\bf{Step 2.}} Claim that when $\rho_{0}(t_{\ast})=0$, we have
\begin{align}\label{C031}
\begin{cases}
u_{m+1}(\cdot,t_{\ast})\equiv0,\quad\mathring{R}_{m+1}(\cdot,t_{\ast})\equiv0,\\\ e(t_{\ast})-\int_{\mathbb{T}^{3}}|u_{m+1}(x,t_{\ast})|^{2}dx\leq\vartheta_{m+2}^{\frac{q-2}{q}}\delta_{m+2}.
\end{cases}
\end{align}
Since $\rho_{0}(t_{\ast})=0$, we have $\rho(t)=0$ for any $|t-t_{\ast}|<\ell$ due to the fact that $\rho$ is nonnegative and continuous. Hence we obtain 
\begin{align*}
\tilde{e}(t)\leq2^{-1}\vartheta_{m+2}^{\frac{q-2}{q}}\delta_{m+2},\quad\forall\,t\in(t_{\ast}-\ell,t_{\ast}+\ell),
\end{align*}
which, in combination with \eqref{ZQ015},  shows that if $0<2b^{2}\beta<\varepsilon$ and $2<q<2+\frac{\varepsilon}{2b(b-1)(1+\varepsilon_{\ast})}$,
\begin{align}\label{F02}
e(t)-\int_{\mathbb{T}^{3}}|u_{m}(x,t)|^{2}dx\leq&	\frac{\vartheta_{m+1}^{\frac{q-2}{q}}\delta_{m+2}}{2}+\vartheta_{m+1}^{\frac{q-2}{q}}\delta_{m+1}\lambda_{m}^{-\varepsilon}\leq\frac{3\vartheta_{m+2}^{\frac{q-2}{q}}\delta_{m+2}}{4}.
\end{align}
It then follows from \eqref{QZ005} that $u_{m}(\cdot,t)\equiv0$ and $\mathring{R}_{m}(\cdot,t)\equiv0$ for all $|t-t_{\ast}|<\ell$, which indicates that $u_{\ell}(\cdot,t_{\ast})\equiv0,\mathring{R}_{\ell}(\cdot,t_{\ast})\equiv0$, and then $\chi_{(i)}(\cdot,t_{\ast})=0$ for every $i\geq1$. This, together with the fact of $\rho_{0}(t_{\ast})=0$, reads that $a_{(\zeta)}(\cdot,t_{\ast})\equiv0$ for any $i\geq0$. So we have $w_{m+1}(\cdot,t_{\ast})\equiv0$ and thus $u_{m+1}(\cdot,t_{\ast})\equiv0$. In light of the fact that $\rho_{0}^{1/2}$ and $\{\chi_{(i)}\}_{i\geq1}$ are all nonnegative smooth functions, then $t_{\ast}$ belongs to their minimum point. Hence we obtain that $\partial_{t}\rho_{0}^{1/2}(t_{\ast})=0$ and $\partial_{t}\chi_{(i)}(\cdot,t_{\ast})\equiv0$ for $i\geq1$, which gives that $\partial_{t}a_{(\zeta)}(\cdot,t_{\ast})\equiv0$ and thus $\partial_{t}w_{m+1}(\cdot,t_{\ast})\equiv0$. A combination of these above facts shows that $\mathring{R}_{m+1}(\cdot,t_{\ast})\equiv0$. Furthermore, we have from \eqref{F01} and \eqref{F02} that
\begin{align*}
&e(t_{\ast})-\int_{\mathbb{T}^{3}}|u_{m+1}(x,t_{\ast})|^{2}dx\notag\\
&=e(t_{\ast})-\int_{\mathbb{T}^{3}}|u_{m}(x,t_{\ast})|^{2}dx+\int_{\mathbb{T}^{3}}|u_{m}(x,t_{\ast})|^{2}dx-\int_{\mathbb{T}^{3}}|u_{\ell}(x,t_{\ast})|^{2}dx\notag\\
&\leq\frac{3\vartheta_{m+2}^{\frac{q-2}{q}}\delta_{m+2}}{4}+\frac{\vartheta_{m+2}^{\frac{q-2}{q}}\delta_{m+2}}{\lambda_{m+1}^{\beta}}\leq\vartheta_{m+2}^{\frac{q-2}{q}}\delta_{m+2}.
\end{align*}		
That is, \eqref{C031} holds. 

{\bf{Step 3.}} Claim that if $\rho_{0}(t_{\ast})\neq0$, then we have
\begin{align}\label{C032}
\frac{\vartheta_{m+2}^{\frac{q-2}{q}}\delta_{m+2}}{4}\leq e(t_{\ast})-\int_{\mathbb{T}^{3}}|u_{m+1}(x,t_{\ast})|^{2}dx\leq\frac{3\vartheta_{m+2}^{\frac{q-2}{q}}\delta_{m+2}}{4}.
\end{align}
On one hand, when $\rho(t_{\ast})>0$, we have from \eqref{W88} that
\begin{align}\label{AZ030}
&e(t_{\ast})-\int_{\mathbb{T}^{3}}|u_{m+1}(x,t_{\ast})|^{2}dx=\tilde{e}(t_{\ast})-3\rho_{0}(t_{\ast})\int_{\mathbb{T}^{3}}\chi_{(0)}^{2}dx+o(1)\vartheta_{m+2}^{\frac{q-2}{q}}\delta_{m+2}\notag\\
&=\frac{\vartheta_{m+2}^{\frac{q-2}{q}}\delta_{m+2}}{2}+3(\rho(t_{\ast})-\rho_{0}(t_{\ast}))\int_{\mathbb{T}^{3}}\chi_{(0)}^{2}dx+o(1)\vartheta_{m+2}^{\frac{q-2}{q}}\delta_{m+2}.
\end{align}	
Analogously as in \eqref{W902}, we obtain that if $0<b^{2}\beta\leq\frac{1}{8}$ and $2<q<2+\frac{1}{4b(b-1)(1+\varepsilon_{\ast})}$,
\begin{align}\label{Z96}
|\rho_{0}(t_{\ast})-\rho(t_{\ast})|\lesssim&\frac{\vartheta_{m+1}^{\frac{q-2}{q}}\delta_{m+1}^{1/2}}{\lambda_{m}^{1/2}}\leq\frac{\vartheta_{m+2}^{\frac{q-2}{q}}\delta_{m+2}}{100|\mathbb{T}^{3}|}.
\end{align}
Inserting this into \eqref{AZ030}, we see that \eqref{C032} holds in the case of $\rho(t_{\ast})>0$.

On the other hand, when $\rho(t_{\ast})=0$, using the fact of $\rho_{0}(t_{\ast})>0$, we obtain that there exists some $t_{0}\in(t_{\ast}-\ell,t_{\ast}+\ell)$ such that 
\begin{align*}
\tilde{e}(t_{0})=\frac{\vartheta_{m+2}^{\frac{q-2}{q}}\delta_{m+2}}{2}.
\end{align*}	
Therefore, it follows from \eqref{W01}, \eqref{W88} and \eqref{Z96} that if $0<2b^{2}\beta<\varepsilon$ and $2<q<2+\frac{\varepsilon}{2b(b-1)(1+\varepsilon_{\ast})}$,
\begin{align*}
&e(t_{\ast})-\int_{\mathbb{T}^{3}}|u_{m+1}(x,t_{\ast})|^{2}dx-\frac{\vartheta_{m+2}^{\frac{q-2}{q}}\delta_{m+2}}{2}\notag\\
&=\tilde{e}(t_{\ast})-\tilde{e}(t_{0})+3(\rho(t_{\ast})-\rho_{0}(t_{\ast}))\int_{\mathbb{T}^{3}}\chi_{(0)}^{2}(x,t_{\ast})dx+o(1)\vartheta_{m+1}^{\frac{q-2}{q}}\delta_{m+2}\notag\\
&\leq \lambda_{m}^{-1}\vartheta_{m+1}^{\frac{q-2}{q}}+\vartheta_{m+1}^{\frac{q-2}{q}}\delta_{m+1}\lambda_{m}^{-\varepsilon}+\frac{\vartheta_{m+2}^{\frac{q-2}{q}}\delta_{m+2}}{20}\leq\frac{\vartheta_{m+2}^{\frac{q-2}{q}}\delta_{m+2}}{4},
\end{align*}	
which shows that \eqref{C032} is proved. 

{\bf{Step 4.}} Proofs of \eqref{EI001}--\eqref{EI002}.  Observe first from \eqref{C031} and \eqref{C032} that \eqref{EI001} holds. For any fixed $t_{\ast}\in[0,T]$, if 
\begin{align*}
e(t_{\ast})-\int_{\mathbb{T}^{3}}|u_{m+1}(x,t_{\ast})|^{2}\leq\frac{\vartheta_{m+2}^{\frac{q-2}{q}}\delta_{m+2}}{10},
\end{align*}	
then we deduce from \eqref{C032} that $\rho_{0}(t_{\ast})=0$. This, in combination with \eqref{C031}, gives that $u_{m+1}(\cdot,t)\equiv0$ and $\mathring{R}_{m+1}(\cdot,t)\equiv0$. The proof is complete.

\end{proof}

We are now ready to prove Proposition \ref{pro01}.
\begin{proof}[Proof of Proposition \ref{pro01}]
A combination of Corollary \ref{coro03}, Lemmas \ref{lem10} and \ref{lem20} shows that Proposition \ref{pro01} holds.

\end{proof}

%\section{Appendix}

%\noindent{\bf{\large Data Availability Statement.}} The data used to support the findings of this study are available from the corresponding author upon request.

\noindent{\bf{\large Acknowledgements.}} This work was partially supported by the National Key research and development program of
China (No. 2022YFA1005700). C. Miao was partially supported by the National Natural Science Foundation of China (No. 12371095).

%\noindent{\bf{\large Statements.}}

%\noindent{\bf{\large Acknowledgements.}}

\end{document}